\documentclass[11pt,reqno]{amsart}
\usepackage{amsmath}
\usepackage{amsfonts}

\numberwithin{equation}{section}
\usepackage{mathrsfs}
\usepackage{times}
\usepackage{amsmath,amsfonts,amstext,amssymb,amsbsy,amsopn,amsthm,eucal}
\usepackage{txfonts}
\usepackage{dsfont}
\usepackage{graphicx}   % for figures
\usepackage{hyperref}
\usepackage{accents}
\usepackage{enumerate}
\usepackage{xcolor}

\usepackage{verbatim}

\newcommand{\Ric}{{\rm Ric}}

\newcommand{\Vol}{{\rm Vol}}

\newtheorem{theorem}{Theorem}[section]

\newtheorem{lemma}[theorem]{Lemma}

\theoremstyle{definition}

\theoremstyle{remark}
\newtheorem{remark}{Remark}[section]
\theoremstyle{remark}

\theoremstyle{remark}

\theoremstyle{remark}
\theoremstyle{remark}

\begin{document}

\title{Rigidity of eigenvalues of shrinking Ricci solitons}
\date{\today}

\author{Chang Li}
\email{chang\_li@pku.edu.cn}
\address{School of Mathematics, Renmin University of China}

\author{Huaiyu Zhang}
\email{zhymath@outlook.com}
\address{School of Mathematics and Statistics, Nanjing University of Science and Technology}

\author{Xi Zhang}
\email{mathzx@njust.edu.cn}
\address{School of Mathematics and Statistics, Nanjing University of Science and Technology}

\begin{abstract}
In this paper, we study the rigidity of eigenvalues of shring Ricci solitons. It is known that the drifted Laplacian on shrinking Ricci solitons has discrete spectrum, its eigenvalues have a lower bound and a rigidity result holds. Firstly, we show that if the $n^\text{th}$ eigenvalue is close to this lower bound, then the $n$-soliton must be the trivial Gaussian soliton $\mathbb{R}^n$. Secondly, we show similar results  for the $(n-1)^\text{th}$ and $(n-2)^\text{th}$ eigenvalue under a non-collapsing condition. Lastly, we give an alomost rigidity for the $k^\text{th}$ eigenvalue with general $k$. Part of our results could be viewed as an soliton (could be noncompact) analog of Theorem 1.1 (which only holds for compact manifolds) in  Peterson (Invent. Math. 138 (1999): 1-21).

\vspace{10pt}

\vspace{10pt}
\end{abstract}

\maketitle
\section{introduction}

A gradient shrinking Ricci soliton is a Riemannian manifold $(M^n,g)$ equipped with a potential function $f$ such that
\begin{align}\label{BEEinstein}
\Ric_f:=\Ric+\nabla^2f=\frac{1}{2}g,
\end{align}
where $\Ric$ is the Ricci curvature of $g$, and $\Ric_f:=\Ric+\nabla^2f$ is named as the Bakry–\'Emery Ricci curvature of $g$ (or $f$-Ricci curvautre of $g$). Ricci solitons are a generalization of the Einstein manifolds in Bakry–\'Emery geometry.

Introduced by Hamilton,  Ricci solitons are important in the theory of Ricci flow (see \cite{Ha93}). They are both the self-similar solutions of Ricci flow and the critical metrics of Perelman's $\mu$-functional. Moreover, Ricci solitons model singularity formation
for the Ricci flow and 
the study of Ricci solitons leads to a better understanding of the singular structures of Ricci flow. Ricci flow has many great applications, such as Poincar\'e conjecture and differential sphere theorem. However, singularity occurs while people are trying to apply Ricci flow in many various geometric problems.   Thus 
 people have great interest to study the geometry of Ricci solitons, see \cite{TZ00}, \cite{Pe03}, \cite{TLZ23}, \cite{Na10}, \cite{TZ12}, \cite{SZ12},   \cite{Br13}, \cite{SU21}, \cite{PW10}, \cite{MW17}, \cite{JWZ}, \cite{CM22}, \cite{CM23b}, etc. and  references therein.

 In the study of solitons, it is known that a suitable  volume element and a correspongding  operator is:
\begin{align}\label{drifted}
d\mu:=e^{-f}d\Vol_g,\ \  \Delta_f u:=\Delta u-\langle \nabla f,\nabla u\rangle, \forall u\in C^\infty(M).
\end{align}
where $d\Vol_g$ is the standard volume element taken with respect to $g$. A direct calculus shows $\Delta_f$ is formally self-adjoint on the space of square-integrable functions on $M$ taken with respect to the volume element $d\mu$.

In fact, a Riemannian manifold $(M,g)$ equipped with a volume element $e^{-f}d\Vol_g$ is also known as a smooth metric measure space. And the operator $\Delta_f$ in \eqref{drifted} is known as the drifted Laplacian (or Witten Laplacian), wich carries important geometric information of the manifold and has attracted great attention. See \cite{CM22}, \cite{CM23}, \cite{Li12}, \cite{LL18}, \cite{CZ17}, \cite{Be20}, etc. for the study of the drifted Laplacian on solitons or smooth metric measure spaces.

On the other hand, rigidity of eigenvalues of manifolds satisfying some curvature conditions is an important topic in geometry. For example, the celebreated Lichnerowicz-Obata  theorem states that if the Ricci curvature of a compact Riemannian manifold satisfies $\Ric\ge (n-1)g$, then the first eigenvalue of Laplace-Bertrami operator satisfies $\lambda_1\ge n$, and the equality holds if and only if the manifold is a round sphere. 

In \cite{Pe99} and \cite{Au05}, Peterson and Aubry extended part of the Lichnerowic-Obata  theorem to an almost rigidity theorem, which states that if the Ricci curvature of a compact Riemannian manifold satisfies $\Ric\ge (n-1)g$, then the manifold is close to a round sphere in Gromov-Hausdorff sense provided that the $n^\text{th}$ eigenvalue $\lambda_n$ of Laplace-Bertrami operator is sufficiently close to $n$. See also \cite{CWZ24} for an analog of Lichnerowicz-Obata  theorem and the almost rigidity theorem for K\"ahler manifolds.

For shrinking Ricci solitons, we can also study the  analog of the Lichnerowicz-Obata  theorem and  the almost rigidity theorem. Since the spectrum of the classical Laplace-Bertrami operator on shrinking Ricci solitons is not discrete in general, the appropriate operator to be considered might be the drifted Laplacian.  By a series of results in \cite{BE85}, \cite{Mo05}, \cite{HN14} and \cite{CZ17}, one has that the drifted Laplacian on a complete gradient shrinking Ricci soliton has discrete spectrum, and the first eigenvalue satisfies $\lambda_1(\Delta_f)\ge \frac{1}{2}$. Moreover, if the first $k^{\text{th}}$ eigenvalues $\lambda_1(\Delta_f),...\lambda_k(\Delta_f)$  are equal to $\frac{1}{2}$, then the soliton splits as a product of some soliton with the Gaussian soliton. The Gaussian soliton is the structure $(\mathbb{R}^{k},\bar g_{\text{Euc};\mathbb{R}^{k}},\bar f_{\text{Euc};\mathbb{R}^{k}})$, here and below $\bar g_{\text{Euc};\mathbb{R}^k}$ denotes the Euclidean metric on $\mathbb{R}^k$, $\bar f_{\text{Gau};\mathbb{R}^k}$ denotes the function $\frac{\sum_{i=1}^k (x^i)^2}{4}+c$ for any $(x^1,...,x^k)\in \mathbb{R}^k$  and $c$ is a constant. It is the most trivial shrinking Ricci solitons. In fact, most  results therein are valid for manifolds satisfies $\Ric_f\ge \frac{1}{2}g$, however we only focus on the soliton case in this paper. See also \cite{GKKO20} for a generalization of these results to RCD spaces.

Since this eigenvalue lower bounds and  rigidity result establishes an analog of the Lichnerowicz-Obata  theorem for shrinking Ricci solitons,  motivated by Peterson and Aubry's almost rigidity theorem in \cite{Pe99} and \cite{Au05} for compact manifolds, in this paper, we  study the almost rigidity for shrinking Ricci solitons.

In fact, better than an almost rigidity result, we have the following gap theorem, which is our first main result:
\begin{theorem}\label{mthmre1}
There exists $\delta=\delta(n)>0$, such that for any complete gradient shrinking Ricci soliton $(M^n,g,f)$, 
if  $\lambda_n(\Delta_f) \le \frac{1}{2}+\delta$, then $(M^n,g,f)$ is identical to $(\mathbb{R}^{n},\bar g_{\text{Euc};\mathbb{R}^n},\bar f_{\text{Gau};\mathbb{R}^n})$.
\end{theorem}

\begin{remark}
It is expectable that one must have the shrinking Ricci soliton  condition, i.e. $\Ric_f= \frac{1}{2}g$, to get a gap theorem like our Theorem \ref{mthmre1}, although  getting the eigenvalues lower bound $\lambda_n(\Delta_f)\ge \lambda_1(\Delta_f)\ge \frac{1}{2}$ only requires $\Ric_f\ge \frac{1}{2}g$. For example, recall the classical volume comparison theorem and Anderson's gap theorem \cite[Lemma 3.1]{An90}, which states that a $n$-manifold with $\Ric=0$ satisfying its asymptotic volume ratio  $\text{AVR}\ge 1-\delta(n)$ must be the Euclidean space. The classical volume comparison theorem gives the asymptotic volume ratio $\text{AVR}\le 1$ and only needs $\Ric\ge 0$, but Anderson's gap theorem needs $\Ric=0$. Thus our condition $\Ric_f=\frac{1}{2}g$ is expectable.
\end{remark}

\begin{remark}
From Theorem \ref{mthmre1}, it follows that any $n$-dimensional complete gradient shrinking  Ricci soliton has $\lambda_{n+1}(\Delta_f)>\frac{1}{2}+\delta(n)$.

To prove that, note that the spectrum of the drifted Laplacian on $(\mathbb{R}^{n},\bar g_{\text{Euc};\mathbb{R}^n},\bar f_{\text{Gau};\mathbb{R}^n})$ is
\begin{align}
\{\frac{i}{2}:i=0,1,...\},
\end{align}
where the eigenvalue $\frac{1}{2}$ has multiplicity $n$. Thus $(\mathbb{R}^{n},\bar g_{\text{Euc};\mathbb{R}^n},\bar f_{\text{Gau};\mathbb{R}^n})$ has $\lambda_{n+1}(\Delta_{\bar f_{\text{Gau};\mathbb{R}^n}})=1$. Therefore, assuming $\delta(n)<\frac{1}{2}$ without loss of generality, by Theorem \ref{mthmre1}, whether $(M^n,g,f)$ is identical to $(\mathbb{R}^{n},\bar g_{\text{Euc};\mathbb{R}^n},\bar f_{\text{Gau};\mathbb{R}^n})$ or not identical to it, $(M^n,g,f)$ must have $\lambda_{n+1}(\Delta_f)>\frac{1}{2}+\delta(n)$.
\end{remark}

We can improve the condition of $\lambda_n(\Delta_f)$ to a condition of $\lambda_{n-1}(\Delta_f)$ or $\lambda_{n-2}(\Delta_f)$ while assuming a non-collapsing condition. Our second main result is:
\begin{theorem}\label{mthmre2}
For any $v>0$, there exists $\delta=\delta(n,v)>0$, such that for any complete gradient shrinking Ricci soliton $(M^n,g,f)$ with
\begin{align}\Vol_g(B_1(p)\ge v,
\end{align}
where $p$ is a minimum point of $f$,
\begin{enumerate}
	\item if  $\lambda_{n-1}(\Delta_f) \le \frac{1}{2}+\delta$, then $(M^n,g,f)$ is identical to $(\mathbb{R}^{n},\bar g_{\text{Euc};\mathbb{R}^n},\bar f_{\text{Gau};\mathbb{R}^n})$

	\item
if  $\lambda_{n-2}(\Delta_f) \le \frac{1}{2}+\delta$, then $(M^n,g,f)$ is identical to  a product soliton $(X,\bar g_X,\bar f_X)\times (\mathbb{R}^{n-2},\bar g_{\text{Euc};\mathbb{R}^{n-2}},\bar f_{\text{Gau};\mathbb{R}^{n-2}}):=(X\times \mathbb{R}^{n-2},\bar g_X+\bar g_{\text{Euc};\mathbb{R}^{n-2}},\bar f_X+\bar f_{\text{Euc};\mathbb{R}^{n-2}})$, 
where $(X,\bar g_X,\bar f_X)$ is $(\mathbb{R}^{2},\bar g_{\text{Euc};\mathbb{R}^2},\bar f_{\text{Gau};\mathbb{R}^2})$, or the standard $\mathbb{S}^2$ or $\mathbb{RP}^2$  with constant $\bar f_X$.

\end{enumerate}
\end{theorem}

At last, for $\lambda_k(\Delta_f)$ with general $k$, we have the following almost rigidity theorem:
\begin{theorem}\label{thm1.2}
For any $r,\epsilon>0$, there exists $\delta=\delta(n,r,\epsilon)>0$, such that for any complete gradient shrinking Ricci soliton $(M^n,g,f)$, if $\lambda_k(\Delta_f) \le \frac{1}{2}+\delta$ for some nonnegative integer $k$, then
\begin{align*}
d_{\text{GH}}(B_{r}(p),B_{r}(\bar p,0^k))<\epsilon,
\end{align*}
for some product length space $(\bar p,0^k)\in X\times \mathbb{R}^k$, where $\mathbb{R}^k$ is equipped with the Euclidean metric, and $p$ is a minimum point of $f$.
\end{theorem}

We will prove Theorem \ref{thm1.2} firstly, and the proof of Theorem \ref{mthmre1} and Theorem \ref{mthmre2} are based on our Theorem \ref{thm1.2} and the strong rigidity of cylindrical shrinkers given by Colding-Minicozzi II (see \cite[Theorem 9.2]{CM22}).

\textbf{Organization:} In section 2, we give some preliminary results for gradient shrinking Ricci solitons. In section 3, we construct a group of eigenfunctions satisfying certain estimates. In section 4, we construct the $\epsilon$-Gromov Hausdorff map based on the eigenfunctions constructed in section 2 and prove our Theorem \ref{thm1.2}. In section 5, we prove our Theorem \ref{mthmre1} and Theorem \ref{mthmre2}.

\section{preliminary results}
In this section we give some preliminary results for gradient shrinking Ricci solitons.

On a gradient shrinking Ricci soliton $(M^n,g,f)$, for the volume form 
$d\mu:=e^{-f}d\Vol_g$ and the drifted Laplacian $\Delta_f$,  by a direct calculus  we  have the following formula of integration by parts:
\begin{align}\label{intbyparts}
\int_M \Delta_f uv d\mu=-\int_M \langle \nabla u,\nabla v\rangle d\mu.
\end{align}

And we also have the following Bochner formula:
\begin{align}\label{Bochner formula}
\frac{1}{2}\Delta_f|\nabla u|^2=|\nabla^2 u|^2+\langle \nabla u,\nabla(\Delta_f u)\rangle +\Ric_f(\nabla u,\nabla u).
\end{align}

 Some results below were initially given for manifolds with more general Bakery-\'Emery Ricci curvature conditions, and we will show how they work for gradient shrinking Ricci solitons.

We begin with the following lemma, which is a minor modification of \cite[Theorem 1.2]{WW09}:
\begin{lemma}[Volume comparison for manifolds with Bakery-\'Emery Ricci curvature bounded below]\label{volcomp12}
Let $(M^n,g,f)$ be a complete Riemannian  manifold such that $\Ric_f\ge-(n-1)\Lambda^2g$. Then for any $0<r\le R$,
\begin{enumerate}
\item if $|\nabla f|\le A$ on the geodesic ball $B_R(p)$, then
\begin{align}
\frac{\mu(B_R(p))}{\mu(B_r(p))}\le e^{AR}\frac{\Vol_\Lambda^n(R)}{\Vol_\Lambda^n(r)},
\end{align}

\item if $|f|\le k$ on $B_R(p)$, then
\begin{align}
\frac{\mu(B_R(p))}{\mu(B_r(p))}\le \frac{\Vol_\Lambda^{n+4k}(R)}{\Vol_\Lambda^{n+4k}(r)},
\end{align}
where $\Vol_\Lambda^n(R)$ denotes the volume of geodesic ball  with radius $r$ in $n$-dimensional complete simply connected
space with constant curvature $-\Lambda^2$.
	\end{enumerate}
\end{lemma}
\begin{remark}
\cite{WW09} only states this result in the case that  $|\nabla f|\le A$ or $|f|\le k$ holds on the whole $M$.  However, the proof only involves the information in $B_R(p)$, thus it actually works for proving Lemma \ref{volcomp12} .

\cite{WW09} also gives results for nonnegative Bakry–\'Emery Ricci curvature lower bounds. However, Lemma \ref{volcomp12} is enough for our use.
\end{remark}

We also need the following lemma, \cite[Lemma 3.3]{WZ19} (see also \cite{Ja15}):
\begin{lemma}[Segment inequality for manifolds with Bakery-\'Emery Ricci curvature bounded below, ]\label{segie1}
Let $(M^n,g,f)$ be a complete Riemannian manifold such that $\Ric_f\ge-(n-1)\Lambda^2g$ and $|f|,|\nabla f|\le A$ for some nonnegative constant $\Lambda$ and $A$. Let $A_1,A_2$ be two subsets of $M$ and $W$ be another subset of $M$ such that $\cup_{y_1\in A_1,y_2\in A_2}\gamma_{y_1y_2}\subset W$, where $\gamma_{y_1y_2}$ is a minimal geodesic curve connecting $y_1$ and $y_2$. Then for any $e\in C^\infty(M)$, there holds
\begin{align}
\int_{A_1\times A_2}\int_0^{d(y_1,y_2)}e(\gamma_{y_1y_2}(s))dsd(\mu\times\mu)\le C(n,\Lambda,A)D\left(\mu(A_1)+\mu(A_2)\right)\int_W ed\mu,
\end{align}
where $\mu\times\mu$ is the product measure, $D=\sup\{d(y_1,y_2):y_1\in A_1,y_2\in A_2\}$.
\end{lemma}

For gradient shrinking Ricci solitons we can get more delicate results. Let us state how the potential function $f$ is normalized firstly, (adding a constant to the potential function does not effect the nature of a soliton). Let $(M^n,g,f)$ be a complete gradient shrinking Ricci soliton, by \cite{Ha93} (see also \cite[Lemma 1.1]{Ca09}), we have
\begin{align}
R+|\nabla f|^2-f=C_1(g),
\end{align}
where $R$ is the scalar curvature of $g$, and $C_1(g)$ is a constant depending only on $g$. Without loss of generality we assume $f$ is normalized (by adding a constant) such that
\begin{align}\label{normf}
R+|\nabla f|^2-f=0.
\end{align}

To apply Lemma \ref{segie1} to gradient shrinking Ricci soliton and get more delicate results, we need the following lemma, which tells us how to compare $f$ with the distance function:
\begin{lemma}[Lemma 2.1 in \cite{HM11}]\label{lm2.4}
Let $(M^n,g,f)$ be a gradient shrinking Ricci soliton with $f$ normalized by satisfying (\ref{normf}), then $f$ attains its infimum. Let $p$ be a minimum point of $f$, we have
\begin{align}\label{fd}
\frac{1}{4}(d(x,p)-C_1(n))^2_{+}\le f(x)\le \frac{1}{4}(d(x,p)+C_2(n))^2,
\end{align}
where $C_i(n),i=1,2$ is a positive constant depending only on $n$, and $a_+:={\text{max}}\{0,a\}$.
\end{lemma}
\begin{remark}
The inequality in \cite[Lemma 2.1]{HM11} looks like being different from \eqref{fd}. In fact, this is because our normalization of $f$ is different from that in \cite{HM11}. From \cite[(1.2) and (2.4)]{HM11}), one can see that our \eqref{fd} follows from \cite[Lemma 2.1]{HM11} given our normalization.
\end{remark}

Then we can apply the lemmas for manifolds with Bakery-\'Emery Ricci curvature bounded below above to gradient shrinking Ricci solitons. We have:
\begin{lemma}[Volume comparison for gradient shrinking Ricci solitons]\label{volcomp}
et $(M^n,g,f)$ be a complete gradient shrinking Ricci soliton, $p$ be a minimum point of $f$. Then for any $0<r\le R$,
\begin{align}
\frac{\mu(B_R(p))}{\mu(B_r(p))}\le C(n,r,R).
\end{align}
\end{lemma}

Lemma \ref{volcomp} follows directly by combining Lemma \ref{volcomp12} (2) and \eqref{fd}.

\begin{lemma}[Segment inequality for gradient shrinking Ricci solitons]\label{segie}
Let $(M^n,g,f)$ be a complete gradient shrinking Ricci soliton, $p$ be a minimum point of $f$. Let $A_1,A_2$ be two subsets of $M$ such that $\cup_{x_1\in A_1,x_2\in A_2}\gamma_{x_1x_2}\subset B_r(q)$, where $q\in B_r(p)$, $r$ is a positive constant and $\gamma_{x_1x_2}$ is a minimal geodesic curve connecting $x_1$ and $x_2$, then for any $e\in C^\infty(M)$, there holds
\begin{align}
\int_{A_1\times A_2}\int_0^{d(x_1,x_2)}e(\gamma_{x_1x_2}(s))dsd(\mu\times\mu)\le C(n,r)\left(\mu(A_1)+\mu(A_2)\right)\int_{B_r(q)} ed\mu,
\end{align}
where $\mu\times\mu$ is the product measure.
\end{lemma}
\begin{proof}
	By \cite[(2.5)]{HM11} we have $R\ge0$ on $M$, thus \eqref{normf} gives
\begin{align}\label{nablaf}
|\nabla f|^2=f-R\le f.
\end{align}

Then by \eqref{nablaf} and Lemma \ref{lm2.4} we have
\begin{align}\label{fcontrol}
0\le |\nabla f|\le f\le C(n,r), \text{on}\ B_r(p), \forall r\ge 0,
\end{align}
where $C(n,r)$ is a nonnegative constant depending only on $n$ and $r$ varies from line to line.

For any  $q\in B_r(p)$, by triangular inequality we have $B_r(q)\subset B_{2r}(p)$.

Thus we have
\begin{align}\label{fcontrol2}
0\le |\nabla f|\le f\le C(n,r), \text{on}\ B_r(q), \forall r\ge 0, q\in B_r(p).
\end{align}

Let $A_1,A_2$ be two subsets of $M$ such that $\cup_{x_1\in A_1,x_2\in A_2}\gamma_{x_1x_2}\subset B_r(q)$, where $q\in B_r(p)$, $r$ is a positive constant and $\gamma_{x_1x_2}$ is a minimal geodesic curve connecting $x_1$ and $x_2$. Then for any $e\in C^\infty(M)$, by Lemma \ref{segie1} and \eqref{fcontrol2}, there holds
\begin{align}\label{segieq3}
\int_{A_1\times A_2}\int_0^{d(y_1,y_2)}e(\gamma_{y_1y_2}(s))dsd(\mu\times\mu)\le C(n,r)D\left(\mu(A_1)+\mu(A_2)\right)\int_{B_r(p)} ed\mu,
\end{align}
where $\mu\times\mu$ is the product measure, $D=\sup\{d(y_1,y_2):y_1\in A_1,y_2\in A_2\}$.

Since $\cup_{x_1\in A_1,x_2\in A_2}\gamma_{x_1x_2}\subset B_r(q)$, by triangular inequality we have $D\le 2r$. Thus \eqref{segieq3} is simplified as
\begin{align}
\int_{A_1\times A_2}\int_0^{d(y_1,y_2)}e(\gamma_{y_1y_2}(s))dsd(\mu\times\mu)\le C(n,r)\left(\mu(A_1)+\mu(A_2)\right)\int_{B_r(p)} ed\mu,
\end{align}
which proves the lemma.
\end{proof}

It is standard to get Poincar\'e inequality and Sobolev inequality from the volume comparison and the segment inequality. Thus by Lemma \ref{volcomp} and Lemma \ref{segie}. we have
\begin{lemma}[Poincar\'e inequality and Sobolev inequality]\label{Poinie}
Let $(M^n,g,f)$ be a complete gradient shrinking Ricci soliton and $p$ be a minimum point of $f$, then for any  $r\ge0$, $q\in B_r(p)$ and $1\le s<n$, we have
\begin{enumerate}
\item Dirichlet Poincar\'e inequality:
\begin{align}
\fint_{B_r(q)}u^sd\mu\le C(n,s,r)\fint_{B_r(q)}|\nabla u|^sd\mu, \forall u\in C_0^\infty(B_r(q)),
\end{align}
\item Neumann Poincar\'e inequality:
\begin{align}
\fint_{B_r(q)}|u-\fint_{B_r(q)}ud\mu|^sd\mu\le C(n,s,r)\fint_{B_r(q)}|\nabla u|^sd\mu, \forall u\in C^\infty(B_r(q)),
\end{align}
\item Dirichlet Sobolev inequality:
\begin{align}
\left(\fint_{B_r(q)}|u|^{s'}d\mu\right)^\frac{1}{s'}\le C(n,s,r)\fint_{B_r(q)}|\nabla u|^sd\mu, \forall u\in C_0^\infty(B_r(q)),
\end{align}
\item Dirichlet Sobolev inequality:
\begin{align}
\left(\fint_{B_r(q)}|u-\fint_{B_r(q)}ud\mu|^{s'}d\mu\right)^\frac{1}{s'}\le C(n,s,r)\fint_{B_r(q)}|\nabla u|^sd\mu, \forall u\in C^\infty(B_r(q)),
\end{align}
where $s'=\frac{ns}{n-s}$.
	\end{enumerate}
\end{lemma}

In this paper $\fint\cdot d\mu$ denote the integral average with respect to the volume element $d\mu$, that is, $\fint_A\cdot d\mu:=\frac{\int_A\cdot d\mu}{\mu(A)}$.

Given Sobolev inequality (Lemma \ref{Poinie} (3)), it is standard to get the following Moser iteration result:
\begin{lemma}[Moser iteration]\label{Msit}
Let $(M^n,g,f)$ be a complete gradient shrinking Ricci soliton, $p$ be a minimum point of $f$, and $q\in B_r(p)$. Suppose $u\in C^\infty (M)$ satisfies $\Delta_f u\ge- Au$ for some nonnegative constant $A$, then for any $r>0$, we have
\begin{align}
\sup_{B_r(p)}u_+\le C(n,r,A)\fint_{B_{2r}(p)} u_+d\mu,
\end{align}
where $u_+$ is the positive part of $u$.
\end{lemma}

Given the formula of integration by parts, \eqref{intbyparts}, it is also standard to get the following Caccioppoli inequality:
\begin{lemma}[Caccioppoli inequality]\label{Caccioppoli inequality}
Let $(M^n,g,f)$ be a complete gradient shrinking Ricci soliton. Suppose $u\in C^\infty (B_{2r}(p))$ satisfies $\Delta_f u\ge- Au$ for some nonnegative constant $A$, then for any $r>0$, we have
\begin{align}
\fint_{B_{r}(p)} |\nabla u|^2d\mu\le C(n,r,A) \fint_{B_{2r}(p)} u^2d\mu.
\end{align}
\end{lemma}

Given Moser iteration (Lemma \ref{Msit}), Caccioppoli inequality (Lemma \ref{Caccioppoli inequality}) and the Bochner formula \eqref{Bochner formula}, it is standard to get the following gradient estimate:
\begin{lemma}[Gradient estimate]\label{Gradestm}
Let $(M^n,g,f)$ be a complete gradient shrinking Ricci soliton. Suppose $u\in C^\infty (B_{2r}(p))$ satisfies $\Delta_f u\ge- Au$ for some nonnegative constant $A$, then for any $r>0$, we have
\begin{enumerate}
\item $\sup_{B_r(p)}|\nabla u|^2\le C(n,r,A)\fint_{B_{\frac{3r}{2}}(p)} |\nabla u|^2d\mu
\le C(n,r,A)\fint_{B_{2r}(p)} u^2d\mu$,
\item $\sup_{B_r(p)}(|\nabla u|^2-1)_+\le C(n,r,A)\fint_{B_{2r}(p)} (|\nabla u|^2-1)_+d\mu$,
	\end{enumerate}
	where $(|\nabla u|^2-1)_+$ is the positive part of $|\nabla u|^2-1$.
\end{lemma}

\section{estimate of eigenfunctions}
In this section, we will construct a group of eigenfunctions satisfying certain estimates. This group of eigenfunctions will play an important role in our construction of the Gromov-Hausdorff map from the soliton to the product space.

The main result in this section is:
\begin{lemma}\label{lm2.1}
There exists a positive constant $C_0(n)$ depending only on $n$, such that for any $r\ge C_0(n),\epsilon>0$, there exists a positive constant $\delta=\delta(n,\epsilon)$ depending only on $\epsilon$ and $n$, such that for any complete gradient shrinking Ricci soliton $(M^n,g,f)$, if $\lambda_k(\Delta_f)\le \frac{1}{2}+\delta$, then there exists eigenfunctions $u_i(i=1,...,k)$ on $M$ corresponding to the eigenvalue $\lambda_i(\Delta_f)$, such that
\begin{enumerate}
\item $r^2\fint_{B_r(p)}|\nabla^2 u_i|^2 d\mu< \epsilon$,
\item $\fint_{B_r(p)}|\langle\nabla u_i,\nabla u_j\rangle-\delta_{ij}| d\mu< \epsilon, \forall i,j=1,...,k$,
\item $\sup_{B_r(p)}|\nabla u_i|<1+\epsilon$,
	\end{enumerate}
	where $p$ is a minimum point of $f$, and $B_r(p)$ is the geodesic ball with radius $r$.
\end{lemma}

To prove Lemma \ref{lm2.1}, we begin with the following lemma:
\begin{lemma}\label{lm2.2}
 Let $(M^n,g,f)$ be a gradient shrinking Ricci soliton, if $\lambda_k(\Delta_f)\le \frac{1}{2}+\delta$ for some $\delta>0$, then there exists eigenfunctions $u_i(i=1,...,k)$ on $M$ corresponding to the eigenvalue $\lambda_i(\Delta_f)$, such that
\begin{enumerate}
	\item $\fint_M  u_i u_j d\mu=\lambda_i(\Delta_f)^{-1}\delta_{ij}$,
	\item $\fint_M \langle\nabla u_i,\nabla u_j\rangle d\mu=\delta_{ij}, \forall i,j=1,...,k$,
\item $\fint_M|\nabla^2 u_i|^2 d\mu\le \delta$,
	\end{enumerate}
\end{lemma}

\begin{remark}
By \cite[Corrolary 4.1]{WW09} (see also \cite{Mo05}), $\mu(M)$ is finite, thus the integral average on $M$ is well defined.
\end{remark}
\begin{remark}\label{rmk3.2}
The functions $u_i(i=1,...,k)$ in Lemma \ref{lm2.2} are just the functions we want in Lemma \ref{lm2.1}. We will prove that they satisfy the conclusions of Lemma \ref{lm2.1} later.
\end{remark}
\begin{proof}
Since $\Delta_f$ is self-adjoint on $M$, we can choose $L^2(M)$-orthonormal  eigenfunctions $\hat u_i\in W^{1,2}(M)(i=1,...,k)$ corresponding to the eigenvalue $\lambda_i(\Delta_f)$ .

Let $u_i=\lambda_i(\Delta_f)^{-\frac{1}{2}}\hat u_i$, then $\fint_M  u_i u_j d\mu=\lambda_i(\Delta_f)^{-1}\delta_{ij}$ and $\Delta_f u_i=-\lambda_i(\Delta_f) u_i$.

Thus by integration by parts, we have
\begin{align*}
\fint_M \langle\nabla u_i,\nabla u_j\rangle d\mu=-\fint_M (\Delta_f u_i)  u_j d\mu=\lambda_i(\Delta_f)\fint_M u_i  u_j d\mu=\delta_{ij},
\end{align*}
which proves (1) and (2).

For (3), by the Bochner formula \eqref{Bochner formula}, we have
\begin{align*}
\frac{1}{2}\Delta_f |\nabla u_i|^2&=|\nabla^2 u_i|^2+\langle \nabla u_i,\nabla(\Delta_f u_i)\rangle+\Ric_f(\nabla u_i,\nabla u_i)\\
&=|\nabla^2 u_i|^2+(\frac{1}{2}-\lambda_i(\Delta_f))|\nabla u_i|^2.
\end{align*}

Integrating it, we have
\begin{align*}
\fint_M|\nabla^2 u_i|^2 d\mu&=(\lambda_i(\Delta_f)-\frac{1}{2})\fint_M|\nabla u_i|^2 d\mu+\frac{1}{2}\fint_M\Delta_f |\nabla u_i|^2 d\mu\\
&=(\lambda_i(\Delta_f)-\frac{1}{2})\fint_M|\nabla u_i|^2 d\mu\\
&\le \delta \cdot1=\delta,
\end{align*}

which completes the proof of the lemma.
\end{proof}

We want to prove Lemma \ref{lm2.1} by using Lemma \ref{lm2.2}. However, there is a difficult that $|\nabla u_i|^2$ might concentrate to the infinity, then its integral average on $B_r(p)$ cannot be estimated by Lemma \ref{lm2.2}, in which the integral average are taken on the whole $M$.  To solve this issue, in   Lemma \ref{lm2.5}-Lemma \ref{5lm3.5} below, we prove that the integral on a ball, large sufficiently and uniformly, is close to the integral on the whole $M$.

\begin{lemma}\label{lm2.5}
Let $(M^n,g,f)$ be a gradient shrinking Ricci soliton, then
\begin{enumerate}
	\item
there exists a positive constant $C_0(n)$ depending only on $n$, such that
\begin{align}
\frac{\mu(B_{C_0(n)}(p))}{\mu(M)}>\frac{1}{2},
\end{align}
\item for any $\epsilon>0$, there exists a positive constant $R_0(n,\epsilon)$ depending only on $\epsilon$ and $n$, such that
\begin{align}
\frac{\mu(M\setminus B_{R}(p))}{\mu(M)}<\epsilon, \forall R\ge R_0(n,\epsilon),
\end{align}

\end{enumerate}
where $p$ is any minimum point of $f$.
\end{lemma}
\begin{proof}
Taking trace of the soliton equation $\Ric+\nabla^2 f=\frac{1}{2}g$, we have
\begin{align}\label{vol1}
R+\Delta f=\frac{n}{2}.
\end{align}

Withous loss of generality we assume   $f$ is normalized such that \eqref{normf} holds. Then combining \eqref{vol1} and \eqref{normf}, we have
\begin{align}
\Delta_f f=\frac{n}{2}-f.
\end{align}

Taking integral and by integration by parts, we have
\begin{align}
0=\int_M\Delta_f fd\mu=\int_M(\frac{n}{2}-f)d\mu.
\end{align}

Thus for any $t\ge 0$, we have
\begin{align}\label{f3.6}
\frac{t^2}{4}\mu (\{f\ge\frac{t^2}{4}\})\le \int_M fd\mu=\int_M\frac{n}{2}d\mu=\frac{n}{2}\mu(M).
\end{align}

By Lemma \ref{lm2.4}, we have
\begin{align}\label{f3.7}
M\setminus B_{t+C_1(n)}(p)\subset \{f\ge\frac{t^2}{4}\}.
\end{align}

Combining \eqref{f3.6} and \eqref{f3.7} we have
\begin{align}\label{ratios}
\frac{\mu(M\setminus B_{t+C_1(n)}(p))}{\mu(M)}\le \frac{2n}{t^2},\forall t\ge 0.
\end{align}

For (1), let $t=3n$ and $C_0(n)=3n+C_1(n)$, we have
\begin{align}
\frac{\mu(M\setminus B_{C_0(n)}(p))}{\mu(M)}< \frac{1}{2}.
\end{align}

Thus we have
\begin{align}
\frac{\mu(B_{C_0(n)}(p))}{\mu(M)}>\frac{1}{2},
\end{align}
which gives (1).

For (2), by \eqref{ratios} we have
\begin{align}\label{ratios1}
\frac{\mu(M\setminus B_t(p))}{\mu(M)}\le \frac{4n}{(t-C_1(n))^2},\forall t\ge C_1(n).
\end{align}

Since $\lim_{t\to +\infty} \frac{4n}{(t-C_1(n))^2}=0$,  \eqref{ratios1} gives (2), thus the lemma is proved.
\end{proof}

\begin{lemma}\label{lmballctl}
Suppose  $u$ is an eigenfunction such that $\Delta_f u+\lambda u=0$, then for any $\epsilon>0$, there exists $R=R(\epsilon,\lambda)$ which does not depend on $u$ and $g$, such that
\begin{align}\label{tpvl1}
\int_{B_R(p)}u^2d\mu\ge (1-\epsilon)\int_M u^2 d\mu.
\end{align}
\end{lemma}
\begin{remark}
Colding-Minicozzi II has established stronger results in \cite{CM22} by using frequency estimates. However, the results here is good enough for our use, and we give a proof here based on a log-Sobolev inequality.
\end{remark}
\begin{proof}
For any gradient shrinking Ricci solitons, we have the following log-Sobolev inequality (see \cite{BE85} or \cite[Inequality (1) and Remark 2]{CZ17}:
\begin{align}\label{logS}
\int_M u^2 \log u^2d\mu \le 4\int_M |\nabla u|^2d\mu,
\end{align}
for any  $u\in W^{1,2}(M)$ such that $\int_M u^2d\mu=\mu(M)$.

Let $u$ be an eigenfunction such that $\Delta_f u+\lambda u=0$. Since \eqref{tpvl1} is invariant while multiplying a constant to $u$, without loss of generality we can assume $\int_M u^2 d\mu=\mu(M)$, for any $K>1$ we define
\begin{align}
\Omega_K:=\{|u|\ge K\}.
\end{align}

By the fact $x\log x\ge -e^{-1}$, we have
\begin{align}\label{ballctl}
\int_{M\setminus\Omega_K}u^2\log u^2d\mu\ge -e^{-1}\mu(M\setminus\Omega_K).
\end{align}

Moreover, by the definition of $\Omega_K$ we have
\begin{align}\label{ballctl2}
\int_{\Omega_K}u^2\log u^2d\mu\ge \log K^2\int_{\Omega_K}u^2d\mu.
\end{align}

On the other hand, by log-Sobolev inequality \eqref{logS} and integration by parts, we have
\begin{align}\label{ballctl3}
\int_M u^2\log u^2d\mu\le 4\int_M |\nabla u|^2d\mu=-4\int_M u\Delta_f ud\mu=4\lambda\int_M u^2d\mu=4\lambda \mu(M).
\end{align}

Combining \eqref{ballctl}-\eqref{ballctl3}, we have
\begin{align}
-e^{-1}\mu(M\setminus\Omega_K)+\log K^2\int_{\Omega_K}u^2d\mu\le 4\lambda \mu(M),
\end{align}
which gives
\begin{align}\label{ballctl4}
\int_{\Omega_K}u^2d\mu\le\frac{4\lambda \mu(M)+e^{-1}\mu(M)}{\log K^2},
\end{align}
and thus
\begin{align}\label{ballctl5}
\int_{M\setminus\Omega_K}u^2d\mu\ge1-\frac{4\lambda \mu(M)+e^{-1}\mu(M)}{\log K^2}.
\end{align}

By Lemma \ref{lm2.5}, for any $K>1$, there exists $R=R(n,K)$, such that
\begin{align}\label{ballctl6}
\mu(M\setminus B_R(p))\le K^{-4}\mu(M).
\end{align}

Then by \eqref{ballctl6} we have
\begin{align}\label{ballctl7}
\int_{(M\setminus\Omega_K)\cap (M\setminus B_R(p))}u^2d\mu\le \mu(M\setminus B_R(p))\cdot K^2\le K^{-2}\mu(M).
\end{align}

Combining \eqref{ballctl5} and \eqref{ballctl7}, we have
\begin{align}\label{ballctl8}
\int_{(M\setminus\Omega_K)\cap  B_R(p)}u^2d\mu\ge \mu(M)-\frac{4\lambda\mu(M)+e^{-1}\mu(M)}{\log K^2} - K^{-2}\mu(M)\ge\mu(M)\left(1-\frac{4\lambda+e^{-1}}{\log K^2}-K^{-2}\right).
\end{align}

For any $\epsilon>0$, we can take $K=K(n,\epsilon,\lambda)$ sufficiently large such that $\left(1-\frac{4\lambda+e^{-1}}{\log K^2}-K^{-2}\right)<\epsilon$. Thus by \eqref{ballctl8}, there exists $R=R(n,\epsilon,\lambda)$ sufficiently large, such that
\begin{align}
\int_{M\setminus B_R(p)}u^2d\mu\ge\int_{(M\setminus\Omega_K)\cap  B_R(p)}u^2d\mu\ge (1-\epsilon) \mu(M)=(1-\epsilon)\int_M u^2 d\mu,
\end{align}
which completes the proof of the lemma.
\end{proof}

\begin{lemma}\label{5lm3.5}
Let $u_i$ be the eigenfunctions given in Lemma \ref{lm2.2}, 
then for any $\epsilon>0$, there exists $R=R(n,\epsilon)$ which does not depend on $u_i$ and $g$, such that
\begin{align}\label{tpvl1z}
\int_{B_R(p)}|\nabla u_i|^2d\mu\ge (1-\epsilon)\int_M |\nabla u_i|^2 d\mu.
\end{align}
\end{lemma}
\begin{proof}
By Lemma \ref{lm2.2} we have $\int_M |\nabla u_i|^2d\mu=\mu(M)$ and $u\in W^{2,2}(M)$, thus we can apply  log-Sobolev inequality \eqref{logS} to $|\nabla u|$ and we get
\begin{align}\label{logS1}
\int_M |\nabla u_i|^2 \log |\nabla u_i|^2d\mu \le 4\int_M |\nabla |\nabla u_i||^2d\mu.
\end{align}

By Kato inequality $|\nabla |\nabla u_i||^2\le |\nabla^2 u_i|^2$ and Lemma \ref{lm2.2} (3), \eqref{logS1} gives
\begin{align}\label{logS2}
\int_M |\nabla u_i|^2 \log |\nabla u_i|^2d\mu \le 4\int_M |\nabla^2 u_i|^2d\mu\le 4\delta.
\end{align}

Similar as the proof of Lemma \ref{lmballctl}, for any $K>1$ we define
\begin{align*}
\Omega_K:=\{|u|\ge K\}.
\end{align*}

By the fact $x\log x\ge -e^{-1}$, we have
\begin{align}\label{ballctlz}
\int_{M\setminus\Omega_K}|\nabla u_i|^2\log |\nabla u_i|^2d\mu\ge -e^{-1}\mu(M\setminus\Omega_K).
\end{align}

Moreover, by the definition of $\Omega_K$ we have
\begin{align}\label{ballctl2z}
\int_{\Omega_K}|\nabla u_i|^2\log |\nabla u_i|^2d\mu\ge \log K^2\int_{\Omega_K}|\nabla u_i|^2d\mu.
\end{align}

Combining \eqref{logS2}-\eqref{ballctl2z}, we have
\begin{align}
-e^{-1}\mu(M\setminus\Omega_K)+\log K^2\int_{\Omega_K}|\nabla u_i|^2d\mu\le 4\delta \mu(M),
\end{align}
which gives
\begin{align}\label{ballctl4z}
\int_{\Omega_K}|\nabla u_i|^2d\mu\le\frac{4\delta \mu(M)+e^{-1}\mu(M)}{\log K^2},
\end{align}
and thus
\begin{align}\label{ballctl5z}
\int_{M\setminus\Omega_K}|\nabla u_i|^2d\mu\ge1-\frac{4\delta \mu(M)+e^{-1}\mu(M)}{\log K^2}.
\end{align}

By Lemma \ref{lm2.5}, for any $K>1$, there exists $R=R(K)$, such that
\begin{align}\label{ballctl6z}
\mu(M\setminus B_R(p))\le K^{-4}\mu(M).
\end{align}

Then by \eqref{ballctl6z} we have
\begin{align}\label{ballctl7z}
\int_{(M\setminus\Omega_K)\cap (M\setminus B_R(p))}|\nabla u_i|^2d\mu\le \mu(M\setminus B_R(p))\cdot K^2\le K^{-2}\mu(M).
\end{align}

Combining \eqref{ballctl5z} and \eqref{ballctl7z}, we have
\begin{align}\label{ballctl8z}
\int_{(M\setminus\Omega_K)\cap  B_R(p)}|\nabla u_i|^2d\mu\ge \mu(M)-\frac{4\delta\mu(M)+e^{-1}\mu(M)}{\log K^2} - K^{-2}\mu(M)\ge\mu(M)\left(1-\frac{4\delta+e^{-1}}{\log K^2}-K^{-2}\right).
\end{align}

Without loss of generality we assume $\delta<1$, then for any $\epsilon>0$, we can take $K=K(n,\epsilon)$ sufficiently large such that $\left(1-\frac{4\delta+e^{-1}}{\log K^2}-K^{-2}\right)<\epsilon$. Thus by \eqref{ballctl8}, there exists $R=R(n,\epsilon)$ sufficiently large, such that
\begin{align}
\int_{M\setminus B_R(p)}|\nabla u_i|^2d\mu\ge\int_{(M\setminus\Omega_K)\cap  B_R(p)}|\nabla u_i|^2d\mu\ge (1-\epsilon) \mu(M)=(1-\epsilon)\int_M |\nabla u_i|^2 d\mu,
\end{align}
which completes the proof of the lemma.
\end{proof}

Using Lemma \ref{lm2.5}-Lemma \ref{5lm3.5}, we can prove that $\{\nabla u_i\}_{i=1,...,k}$ is ``almost orthonormal" in integral average sense on $B_R(p)$ for $R$ large:
\begin{lemma}\label{alor}
For any $\epsilon>0$, there exists a positive constant $R_1(n,\epsilon)$ depending only on $\epsilon$ and $n$, such that if $R\ge R_1(n,\epsilon)$, then
\begin{align*}
\left|\fint_{B_R(p)} \left(\langle \nabla u_i,\nabla u_j\rangle-\delta_{ij}\right) d\mu\right|<\epsilon, \forall i,j=1,...,k,
\end{align*}
where $u_i(i=1,...,k)$ are the functions constructed in Lemma \ref{lm2.2}.
\end{lemma}

Before proving Lemma \ref{alor}, let us give a notation. We denote by
\begin{align}
\Psi=\Psi(\epsilon|c_1,...,c_N),
\end{align}
some nonnegative functions such that for any fixed $c_1,...,c_N$,
\begin{align}\label{Psi}
\lim_{\epsilon\to 0^+}\Psi(\epsilon|c_1,...,c_N)=0.
\end{align}

Once the parameters $\epsilon,c_1,...,c_N$ have been fixed, we allow the specific function $\Psi$ to change from line to line.

Thus the conclusion of Lemma \ref{alor} could be rephrased as:
\begin{align}
\left|\fint_{B_R(p)} \left(\langle \nabla u_i,\nabla u_j\rangle-\delta_{ij}\right) d\mu\right|
\le \Psi(R^{-1}|n),\forall i,j=1,...,k,
\end{align}
\begin{proof}
By Lemma \ref{5lm3.5}, we have
\begin{align}
\int_{B_R(p)} |\nabla u_i|^2d\mu \ge \mu(M)\left(1-\Psi(R^{-1}|n)\right),\forall i=1,...,k.
\end{align}

Thus we have
\begin{align}
\int_{M\setminus B_R(p)} |\nabla u_i|^2d\mu \le \mu(M)\Psi(R^{-1}|n),\forall i=1,...,k.
\end{align}

Then we have
\begin{align}\label{MBR}
\int_{M\setminus B_R(p)} |\langle \nabla u_i,\nabla u_j\rangle|d\mu \le
2\int_{M\setminus B_R(p)} (|\nabla u_i|^2+|\nabla u_j|^2)d\mu \le \mu(M)\Psi(R^{-1}|n),\forall i,j=1,...,k.
\end{align}

Combining \eqref{MBR} and Lemma \ref{lm2.2} (2), we have
\begin{align}
\left|\int_{B_R(p)} \langle \nabla u_i,\nabla u_j\rangle d\mu-\mu(M)\delta_{ij}\right| &=
\left|\int_{B_R(p)} \langle \nabla u_i,\nabla u_j\rangle d\mu-\int_{M} \langle \nabla u_i,\nabla u_j\rangle d\mu\right| \notag\\
&\le\int_{M\setminus B_R(p)} |\langle \nabla u_i,\nabla u_j\rangle|d\mu \notag\\
&\le \mu(M)\Psi(R^{-1}|n),\forall i,j=1,...,k.
\end{align}

Then we have
\begin{align}
\left|\int_{B_R(p)} \left(\langle \nabla u_i,\nabla u_j\rangle-\delta_{ij}\right) d\mu\right|
&=
\left|\int_{B_R(p)} \langle \nabla u_i,\nabla u_j\rangle d\mu-\mu(B_R)\delta_{ij}\right| \notag\\
&\le \left|\int_{B_R(p)} \langle \nabla u_i,\nabla u_j\rangle d\mu-\mu(M)\delta_{ij}\right|+\mu(M\setminus B_R(p))\delta_{ij}\notag\\
&\le \mu(M)\Psi(R^{-1}|n)+\mu(M\setminus B_R(p))\delta_{ij},\forall i,j=1,...,k.
\end{align}

Thus we have
\begin{align}\label{ABR1}
\left|\fint_{B_R(p)} \left(\langle \nabla u_i,\nabla u_j\rangle-\delta_{ij}\right) d\mu\right|
&\le \frac{\mu(M)}{\mu(B_R(p))}\left(\Psi(R^{-1}|n)+\frac{\mu(M\setminus B_R(p))}{\mu(M)}\delta_{ij}\right),\forall i,j=1,...,k.
\end{align}

By Lemma \ref{lm2.5}, we have
\begin{align}\label{ABR2}
\frac{\mu(M)}{\mu(B_R(p))}\le\frac{1}{2}, \forall R>C_0(n),
\end{align}
\begin{align}\label{ABR3}
\frac{\mu(M\setminus B_R(p))}{\mu(M)}\le\Psi(R^{-1}|n).
\end{align}

Combining, \eqref{ABR1}-\eqref{ABR3}, we have
\begin{align}
\left|\fint_{B_R(p)} \left(\langle \nabla u_i,\nabla u_j\rangle-\delta_{ij}\right) d\mu\right|
\le \Psi(R^{-1}|n),\forall i,j=1,...,k,
\end{align}
which proves the lemma.
\end{proof}

Now we can prove Lemma \ref{lm2.1}:
\begin{proof}[Proof of Lemma \ref{lm2.1}]

For any $\delta>0$, let $u_i(i=1,...,k)$ be the eigenfunctions constructed in Lemma \ref{lm2.2}.

For (1), by Lemma \ref{lm2.2} we have $\fint_M|\nabla^2 u_i|^2 d\mu< \delta$, and then we have
\begin{align}
\fint_{B_r(p)}|\nabla^2 u_i|^2 d\mu\le \frac{\mu(M)}{\mu(B_r(p))} \fint_M|\nabla^2 u_i|^2 d\mu< \delta \frac{\mu(M)}{\mu(B_r(p))}.
\end{align}

By Lemma \ref{lm2.5}, for any $r\ge C_0(n)$, we have
\begin{align}
\fint_{B_r(p)}|\nabla^2 u_i|^2 d\mu\le \frac{\mu(M)}{\mu(B_r(p))} \mu< 2\delta,
\end{align}
which proves (1).

For (2), let $w=\langle \nabla u_i,\nabla u_j\rangle-\delta_{ij}$, then by Lemma \ref{alor} we have
\begin{align}
\left|\fint_{B_R(p)}w d\mu\right|<\Psi(R^{-1}|n).
\end{align}

By the Neumann Poincar\'e inequality (Lemma \ref{Poinie} (2)), we have
\begin{align}\label{Pciq1}
\fint_{B_R(p)}\left|w-\fint_{B_R(p)}w d\mu\right|d\mu\le C(n,R)\fint_{B_R(p)} |\nabla w|d\mu.
\end{align}

By Cauchy inequality, we have
\begin{align}\label{Pciq2}
|\nabla w|\le |\nabla^2u_i||\nabla u_j|+|\nabla^2u_j||\nabla u_i|.
\end{align}

Since $\Delta_f u_i=-\lambda_i(\Delta_f) u_i$ where $\lambda_i(\Delta_f)\in [\frac{1}{2},1)$, by the gradient estimate (Lemma \ref{Gradestm} (1)), we have
\begin{align}\label{Pciq31}
\sup_{B_R(p)}|\nabla u_i|\le C(n,R)\fint_{B_{2R}(p)}|\nabla u_i|^2d\mu,
\end{align}

By Lemma \ref{lm2.5} (1), we have
\begin{align}\label{Pciq32}
\fint_{B_{2R}(p)}|\nabla u_i|^2d\mu\le \frac{\mu(M)}{\mu(B_{2R}(p))}\fint_M |\nabla u_i|^2d\mu\le 2\fint_M |\nabla u_i|^2d\mu, \forall R\ge C_0(n).
\end{align}

Combining \eqref{Pciq31}, \eqref{Pciq32} and Lemma \ref{lm2.2} (2), we have
\begin{align}\label{Pciq3}
\sup_{B_R(p)}|\nabla u_i|\le C(n,R), \forall R\ge 0.
\end{align}

Combining \eqref{Pciq1}, \eqref{Pciq2} and \eqref{Pciq3}, we have
\begin{align}\label{Pciq0}
\fint_{B_R(p)}\left|w-\fint_{B_R(p)}w d\mu\right|d\mu\le C(n,R)\fint_{B_R(p)} (|\nabla^2u_i|+|\nabla^2u_j|)d\mu.
\end{align}

By Lemma \ref{lm2.1} (1) proved above, we have
\begin{align}
\fint_{B_R(p)}|\nabla^2 u_i|^2 d\mu< \Psi(\delta|R,n), \forall i=1,...,k.
\end{align}

By H\"older inequality, we have
\begin{align}\label{Pciq4}
\fint_{B_R(p)}|\nabla^2 u_i| d\mu\le \left(\fint_{B_R(p)}|\nabla^2 u_i|^2 d\mu\right)^\frac{1}{2}, \forall i=1,...,k.
\end{align}

Combining \eqref{Pciq0}-\eqref{Pciq4}, we have
\begin{align}\label{Pciq5}
\fint_{B_R(p)}\left|w-\fint_{B_R(p)}w d\mu\right|d\mu\le \Psi(\delta|R,n).
\end{align}

By triangle inequality, we have
\begin{align}\label{Pciq6}
|w|\le \left|w-\fint_{B_R(p)}w d\mu\right|+\left|\fint_{B_R(p)}w d\mu\right|.
\end{align}

By Lemma \ref{alor}, we have
\begin{align}\label{Pciq7}
\fint_{B_R(p)}\left|\fint_{B_R(p)}w d\mu\right|d\mu&=\left|\fint_{B_R(p)}w d\mu\right|\fint_{B_R(p)}d\mu\notag\\
&=\left|\fint_{B_R(p)}w d\mu\right|\notag\\
&<\Psi(R^{-1}|n).
\end{align}

Combining \eqref{Pciq5}-\eqref{Pciq7}, we have
\begin{align}\label{Pciq8}
\fint_{B_R(p)}|w| d\mu
\le\fint_{B_R(p)}\left|w-\fint_{B_R(p)}w d\mu\right|d\mu+\fint_{B_R(p)}\left|\fint_{B_R(p)}w d\mu\right|d\mu\le \Psi(\delta|R,n)+\Psi(R^{-1}|n).
\end{align}

Thus for any $\epsilon>0$, we can choose a sufficiently large $R_2=R_2(n,\epsilon)$ depending only on $\epsilon$ and $n$, such that the $\Psi(R^{-1}|n)$ in \eqref{Pciq8} satisfies $\Psi(R^{-1}|n)<\epsilon/4$.

Then we choose a small $\delta_1(n,\epsilon)$ depending only on $\epsilon$ and $n$, such that for any $\delta\in (0,\delta_1]$ the $\Psi(\delta|R_2,n)$ in \eqref{Pciq8} satisfies $\Psi(\delta|R_2,n)<\epsilon/4$.

Therefore we have
\begin{align}\label{Pciq9}
\fint_{B_{R_2}(p)}|w| d\mu<\frac{\epsilon}{2}, \forall \delta\in (0,\delta_1].
\end{align}

For any $r\in [C_0(n),R_2(n,\epsilon)]$, where the $C_0(n)$ is given in Lemma \ref{lm2.5}, we have
\begin{align}\label{Pciq10}
\fint_{B_r(p)}|w| d\mu\le \frac{\mu(B_{R_2}(p))}{\mu(B_r(p))}\fint_{B_{R_2}(p)}|w| d\mu.
\end{align}

By Lemma \ref{lm2.5}, for any $r\ge C_0(n)$, we have
\begin{align}\label{Pciq11}
\frac{\mu(B_{R_2}(p))}{\mu(B_r(p))}<\frac{\mu(M)}{\mu(B_r(p))}<2.
\end{align}

Combining \eqref{Pciq9}-\eqref{Pciq11}, we have
\begin{align}\label{pp34}
\fint_{B_r(p)}|w| d\mu<\epsilon,  \forall \delta\in (0,\delta_1(n,\epsilon)], r\in [C_0(n),R_2(n,\epsilon)].
\end{align}

For any $r\in [C_0(n),R_2(n,\epsilon)]$, we choose $\epsilon$ sufficiently small such that $r\in [C_0(n),R_2(n,\epsilon)]$. Thus we can take $\delta_0=\delta(n,r,\epsilon)$ sufficiently small, such that
\begin{align}\label{pp35}
\fint_{B_r(p)}|w| d\mu<\epsilon,  \forall \delta\in (0,\delta_0(n,r,\epsilon)].
\end{align}
which proves (2).

For (3), since $\Delta_f u_i=-\lambda_i(\Delta_f) u_i$ where $\lambda_i(\Delta_f)\in [\frac{1}{2},1)$, by the gradient estimate (Lemma \ref{Gradestm} (2)), we have
\begin{align}\label{Pciq35}
\sup_{B_r(p)}(|\nabla u_i|^2-1)_+\le C(n,r,A)\fint_{B_{2r}(p)} (|\nabla u_i|^2-1)_+d\mu\le C(n,r,A)\fint_{B_{2r}(p)} \left||\nabla u_i|^2-1\right|d\mu.
\end{align}

Let $j=i$, then $w=|\nabla u_i|^2-1$, then by Lemma \ref{lm2.1} (2) proved above, for any $\epsilon>0$, there exists $\delta_1=\delta_1(n,\epsilon)$ depending only on $\epsilon$ and $n$, such that
\begin{align}\label{Pciq34}
\fint_{B_{2r}(p)}\left||\nabla u_i|^2-1\right|d\mu<\epsilon, \forall \delta\in(0,\delta_1], r\ge C_0(n).
\end{align}

By \eqref{Pciq35} and \eqref{Pciq34}, we have for any $\epsilon>0$, there exists $\delta_0=\delta_0(n,\epsilon)$ depending only on $\epsilon$ and $n$, such that
\begin{align}
\sup_{B_r(p)} |\nabla u_i| <1+\epsilon, \forall \delta\in(0,\delta_0], r\ge C_0(n),
\end{align}
which gives (3), thus the lemma is proved.
\end{proof}

\section{proof of Theorem \ref{thm1.2}}
In this section, we prove Theorem \ref{thm1.2} following the approach in \cite{CC96}, see also \cite{CJN21} and \cite{WZ19}. Let us reclaim it again here:
\begin{theorem}\label{thmap}
For any $r,\epsilon>0$, there exists a $\delta=\delta(n,r,\epsilon)>0$, such that for any complete gradient shrinking Ricci soliton $(M^n,g,f)$, if $\lambda_k(\Delta_f)\le \frac{1}{2}+\delta$ for some $\delta>0$ for some nonnegative integral $k$, then
\begin{align*}
d_{\text{GH}}(B_{r}(p),B_{r}(\bar p,0^k))<\epsilon,
\end{align*}
for some product length space $(\bar p,0^k)\in X\times \mathbb{R}^k$.
\end{theorem}

Let $(M^n,g,f)$ be a complete gradient shrinking Ricci soliton and $p$ be a minimum point of $f$. Suppose $\lambda_k(\Delta_f)\le \frac{1}{2}+\delta$ for some nonnegative integral $k$. Then by Lemma \ref{lm2.1}, for any $r\ge C_0(n)$, there exists $u_i(i=1,...,k)$ on $B_{100r}(p)$ such that

\begin{enumerate}
	\item[(u1)] $(100r)^2\fint_{B_{100r}(p)}|\nabla^2 u_i|^2 d\mu< \Psi(\delta|r,n)$,
\item[(u2)] $\fint_{B_{100r}(p)}|\langle\nabla u_i,\nabla u_j\rangle-\delta_{ij}| d\mu< \Psi(\delta|r,n), \forall i,j=1,...,k$,
\item[(u3)] $\sup_{B_{100r}(p)}|\nabla u_i|<1+\Psi(\delta|r,n)$,
\item[(u4)] $\Delta_f u_i=-\lambda_i(\Delta_f)u_i$, where $\lambda_i(\Delta_f)\in [\frac{1}{2},\frac{1}{2}+\delta]$.
	\end{enumerate}

	For the definition of the notation $\Psi$, see \eqref{Psi}.

We denote $\mathbf{u}:=(u_1,...,u_k)$ and $X:=\mathbf{u}^{-1}(\{\mathbf{u}(p)\})$ (which is the inverse image of $\{\mathbf{u}(p)\}\subset \mathbb{R}^k$). For any $x\in B_r(p)$, we can choose a $\bar x\in X$ such that $d(x,\bar x)=d(x,X)$. Define a map
\begin{align}
F:B_r(p)&\to X\times \mathbb{R}^k,\notag\\
x&\mapsto (\bar x,\mathbf{u}(x)).
\end{align}

We will prove that $F$ is a $\Psi(\delta|r,n)$-Gromov Hausdorff map from $B_r(p)$ to $B_r(\bar p, \mathbf{u}(p))$.

 Firstly we have the following quantitative version
of the Pythagorean theorem:
\begin{lemma}\label{Pyth}
Let $(M^n,g,f)$ be a complete gradient shrinking Ricci soliton and $p$ be a minimum point of $f$. For any $u\in C^\infty(B_{100r}(p))$ satisfying (u1)-(u3), let $x,y,z\in B_r(p)$ such that $|u(x)-u(z)|<\delta$, $|d(y,z)-\left|u(y)-u(z)|\right|<\delta$, then there holds
\begin{align}
\left|d^2(x,z)+d^2(y,z)-d^2(x,y)\right|<\Psi(\delta|r,n).
\end{align}
\end{lemma}

We will use the following segment inequality in the proof of Lemma \ref{Pyth}:
\begin{lemma}\label{lmseg}
Let $(M^n,g,f)$ be a gradient shrinking Ricci soliton, $p$ be a minimum point of $f$, then for any $e\in C^\infty(M)$, there holds
\begin{enumerate}
	\item
\begin{align}
\fint_{B_{2r}(p)\times B_{2r}(p)}\int_0^{d(x_1,x_2)}e(\gamma_{x_1x_2}(s))dsd(\mu\times\mu)\le  C(n,r)\fint_{B_{6r}(p)} ed\mu,
\end{align}
\item
\begin{align}
\fint_{B_{2r}(p)\times B_{2r}(p)\times B_{2r}(p)}\int_0^{d(x_1,x_2)}\int_0^{d(x_1,\gamma_{x_2x_3}(t))} e(\gamma_{x_1\gamma_{x_2x_3}(t)}(s))dsdtd(\mu\times\mu\times\mu)\le  C(n,r)\fint_{B_{6r}(p)} ed\mu,
\end{align}
\end{enumerate}
where $\gamma_{x_1 x_2}$ is a minimal geodesic from $x_1$ to $x_2$ and $\mu\times\mu$ is the product measure.
\end{lemma}

\begin{proof}
For (1),
Let $A_1=A_2=B_2(p)$, then $D:=\sup\{d(x_1,x_2):x_1\in A_1,x_2\in A_2\}=4r$.

And for any $x_3\in \gamma_{x_1x_2}$, by triangle inequality we have
\begin{align}
d(x_1,x_3)\le d(x_1,x_2)\le d(p,x_1)+d(p,x_2)<2r+2r=4r,
\end{align}
\begin{align}
d(p,x_3)\le d(p,x_1)+d(x_1,x_3)<2r+4r=6r.
\end{align}

Thus we have $\cup_{x_1\in A_1,x_2\in A_2}\gamma_{x_1x_2}\subset B_{6r}(p)$, and by Lemma \ref{segie}, we have
\begin{align}\label{segapp}
\int_{B_{2r}(p)\times B_{2r}(p)}\int_0^{d(x_1,x_2)}e(\gamma_{x_1x_2}(s))dsd(\mu\times\mu)&\le C(n,r)D\left(\mu(B_{2r}(p))+\mu(B_{2r}(p))\right)\int_{B_{6r}(p)} ed\mu\notag\\
&\le C(n,r)\mu(B_{2r}(p)\int_{B_{6r}(p)} ed\mu,
\end{align}

Thus by \eqref{segapp} we have
\begin{align}\label{segapp2}
\fint_{B_{2r}(p)\times B_{2r}(p)}\int_0^{d(x_1,x_2)}e(\gamma_{x_1x_2}(s))dsd(\mu\times\mu)&=\frac{1}{\mu(B_{2r}(p)^2}\int_{B_{2r}(p)\times B_{2r}(p)}\int_0^{d(x_1,x_2)}e(\gamma_{x_1x_2}(s))dsd(\mu\times\mu)\notag\\
 &\le C(n,r)\frac{1}{\mu(B_{2r}(p)}\int_{B_{6r}(p)} ed\mu\notag\\
 &= C(n,r)\frac{\mu(B_{6r}(p)}{\mu(B_{2r}(p)}\fint_{B_{6r}(p)} ed\mu.
\end{align}

Without loss of generality we assume $r>C_0(n)$ for the $C_0(n)$ in Lemma \ref{lm2.5}, then by Lemma \ref{lm2.5} (1), we have
\begin{align}\label{segapp1}
\frac{\mu(B_{6r}(p)}{\mu(B_{2r}(p)}\le \frac{\mu(B_{M}(p)}{\mu(B_{2r}(p)}\le \frac{1}{2}.
\end{align}

By \eqref{segapp2} and \eqref{segapp1}, we have
\begin{align}
\fint_{B_{2r}(p)\times B_{2r}(p)}\int_0^{d(x_1,x_2)}e(\gamma_{x_1x_2}(s))dsd(\mu\times\mu)\le  C(n,r)\fint_{B_{6r}(p)} ed\mu,
\end{align}
which proves (1).

For (2), using (1) twice we have
\begin{align*}
&\fint_{B_{2r}(p)\times B_{2r}(p)\times B_{2r}(p)}\int_0^{d(x_1,x_2)}\int_0^{d(x_1,\gamma_{x_2x_3}(t))} e(\gamma_{x_1\gamma_{x_2x_3}(t)}(s))dsdtd(\mu\times\mu\times\mu)\\
\le &\fint_{B_{2r}(p)\times B_{2r}(p)}\int_0^{d(x_1,x_4)} e(\gamma_{x_1x_4}(s))dsd(\mu\times\mu)\\
\le & C(n,r)\fint_{B_{6r}(p)} ed\mu,
\end{align*}
which proves (2). Thus the lemma is proved.
\end{proof}

Now we prove Lemma \ref{Pyth}:
\begin{proof}[Proof of Lemma \ref{Pyth}]
Let $e=|\nabla^2 u|^2+\left||\nabla u|^2-1\right|$. By Lemma \ref{lmseg} we have
\begin{align}\label{segctrle1}
\fint_{B_{2r}(p)\times B_{2r}(p)}\int_0^{d(x_1,x_2)}e(\gamma_{x_1x_2}(s))dsd(\mu\times\mu)\le  C(n,r)\fint_{B_{6r}(p)} ed\mu,
\end{align}
\begin{align}\label{segctrle2}
\fint_{B_{2r}(p)\times B_{2r}(p)\times B_{2r}(p)}\int_0^{d(x_1,x_2)}\int_0^{d(x_1,\gamma_{x_2x_3}(t))} e(\gamma_{x_1\gamma_{x_2x_3}(t)}(s))dsdtd(\mu\times\mu\times\mu)\le  C(n,r)\fint_{B_{6r}(p)} ed\mu.
\end{align}
The right hand side of inequalities above is controlled by
\begin{align}\label{segctrle3}
\fint_{B_{6r}(p)} ed\mu\le \frac{\mu(B_{100r}(p)}{\mu(B_{6r}(p)}\fint_{B_{100r}(p)} ed\mu.
\end{align}

Without loss of generality we assume $r>C_0(n)$ for the $C_0(n)$ in Lemma \ref{lm2.5}, then by Lemma \ref{lm2.5} (1), we have
\begin{align}\label{segctrle4}
\frac{\mu(B_{100r}(p)}{\mu(B_{6r}(p)}\le \frac{\mu(B_{M}(p)}{\mu(B_{6r}(p)}\le \frac{1}{2}.
\end{align}

Combining (u1), (u2) and \eqref{segctrle1}-\eqref{segctrle4}, we have
\begin{align}\label{segctrle5}
\fint_{B_{2r}(p)\times B_{2r}(p)}\int_0^{d(x_1,x_2)}e(\gamma_{x_1x_2}(s))dsd(\mu\times\mu)< \Psi(\delta|r,n),
\end{align}
and
\begin{align}\label{segctrle6}
\fint_{B_{2r}(p)\times B_{2r}(p)\times B_{2r}(p)}\int_0^{d(x_1,x_2)}\int_0^{d(x_1,\gamma_{x_2x_3}(t))} e(\gamma_{x_1\gamma_{x_2x_3}(t)}(s))dsdtd(\mu\times\mu\times\mu)<  \Psi(\delta|r,n).
\end{align}

For any $x,y,z$ satisfying the assumption of the lemma, by \eqref{segctrle5}, we can take $\tilde x,\tilde y,\tilde z$ in $B_{\Psi(\delta|r,n)}(x),B_{\Psi(\delta|r,n)}(y),B_{\Psi(\delta|r,n)}(z)$ respectively, such that
\begin{align}\label{nabla u 2}
\int_0^{d(\tilde y, \tilde z)}\left||\nabla u|^2-1\right|(\sigma(s))ds\le \Psi(\delta|r,n),
\end{align}
where $\sigma:[0,d(\tilde y, \tilde z)]\to B_{100r}(p)$ is the unique minimal geodesic with unit velocity from $\tilde y$ to $\tilde z$.

Moreover, by \eqref{segctrle6}, $\tilde y,\tilde z$ could be chosen such that there exists a full measure open subset $U\subset [0,d(\tilde y, \tilde z)]$, such that for any $s\in U$ there is a unique minimal geodesic  with unit velocity  $\tau_s:[0,l(s)]\to B_{100r}(p)$ from $\tilde x$ to $\sigma(s)$, where $l(s):=d(\tilde x,\sigma(s))$, and
\begin{align}\label{hessian estm}
\int_U\int_0^{l(s)}|\nabla^2 u(\tau_s(t))|dtds\le \Psi(\delta|r,n).
\end{align}

We want to prove that $\nabla u$ is close to $\sigma'$ along $\sigma$ in integral sense. Firstly we estimate their inner product in integral sense. Without loss of generality we assume $u(z)>u(y)=0$. By the fundamental theorem of calculus we have
\begin{align}\label{Jlm3.7eq1}
\left|\int_0^{d(\tilde y, \tilde z)}\left(\langle \nabla u,\sigma'(s)\rangle-1\right)ds\right|
&=\left|\int_0^{d(\tilde y, \tilde z)}\left((u\circ\sigma)'(s)-1\right)ds\right|\\\notag
&=|u(\tilde z)-u(\tilde y)-d(\tilde y, \tilde z)|
\end{align}

Recall that $\tilde y,\tilde z$  belong to $B_{\Psi(\delta|r,n)}(y),B_{\Psi(\delta|r,n)}(z)$ respectively and in the assumption of the lemma we have $|d(y,z)-\left|u(y)-u(z)|\right|<\delta$, thus we have
\begin{align}\label{Jlm3.7eq2}
|u(\tilde z)-u(\tilde y)-d(\tilde y, \tilde z)|\le \Psi(\delta|r,n).
\end{align}

By \eqref{Jlm3.7eq1} and \eqref{Jlm3.7eq2} we have
\begin{align}\label{Jlm3.7eq3}
\left|\int_0^{d(\tilde y, \tilde z)}\left(\langle \nabla u,\sigma'(s)\rangle-1\right)ds\right|
\le \Psi(\delta|r,n).
\end{align}

To estimate the difference between $\nabla u$ and $\sigma'$ in integral sense, we calculate that
\begin{align}\label{Jlm3.7eq4}
\int_0^{d(\tilde y, \tilde z)}|\nabla u(\sigma(s))-\sigma'(s)|^2ds
&= \int_0^{d(\tilde y, \tilde z)}|\nabla u(\sigma(s))|^2+|\sigma'(s)|^2-2\langle \nabla u(\sigma(s)),\sigma'(s)\rangle ds\notag\\
&\le \int_0^{d(\tilde y, \tilde z)}\left||\nabla u(\sigma(s))|^2-1\right| ds
+\int_0^{d(\tilde y, \tilde z)}\left||\sigma'(s)|^2-1\right|ds
+\int_0^{d(\tilde y, \tilde z)}
|2-2\langle \nabla u(\sigma(s)),\sigma'(s)\rangle| ds.
\end{align}

Combining \eqref{nabla u 2}, \eqref{Jlm3.7eq3} and \eqref{Jlm3.7eq4},  and recall that $|\sigma'|\equiv 1$, we have
\begin{align}\label{Jlm3.7eq5}
\int_0^{d(\tilde y, \tilde z)}|\nabla u(\sigma(s))-\sigma'(s)|^2ds
\le \Psi(\delta|r,n).
\end{align}

Now we estimate $d(y,z)$ in terms of $u$,
\begin{align}\label{Jlm3.7eq6}
\frac{1}{2}d(y,z)
&=d(\tilde y,\tilde z)\pm\Psi(\delta|r,n)\notag\\
&=\int_0^{d(\tilde y,\tilde z)}sds\pm\Psi(\delta|r,n).
\end{align}

To estimate the integral above in terms of $u$, by the fundamental theorem of calculus we have
\begin{align}\label{Jlm3.7eq7}
|u(\sigma(s))-u(\sigma(0))-s|
=\left|\int_0^s\left(\langle \nabla u,\sigma'(s')\rangle-1\right)ds'\right|.
\end{align}

To estimate $\left|\int_0^s\left(\langle \nabla u,\sigma'(s)\rangle-1\right)ds'\right|$, by H\"older inequality and (u3) we have
\begin{align}|\nabla u,\sigma'(s)|\le |\nabla u|\le 1+\Psi(\delta|r,n).
\end{align}

Thus we have
\begin{align}\label{Jlm3.7eq8}
\left|\int_0^{d(\tilde y,\tilde z)}\left(\langle \nabla u,\sigma'(s)\rangle-1\right)_+ds\right|\le \Psi(\delta|r,n),
\end{align}
where $\left(\langle \nabla u,\sigma'(s)\rangle-1\right)_+$ is the positive part of $\left(\langle \nabla u,\sigma'(s)\rangle-1\right)$

By \eqref{Jlm3.7eq1},  we have
\begin{align}\label{Jlm3.7eq9}
\left|\int_0^{d(\tilde y,\tilde z)}\left(\langle \nabla u,\sigma'(s)\rangle-1\right)_+-\left(\langle \nabla u,\sigma'(s)\rangle-1\right)_-ds\right|
=\left|\int_0^{d(\tilde y,\tilde z)}\left(\langle \nabla u,\sigma'(s)\rangle-1\right)ds\right|
\le \Psi(\delta|r,n),
\end{align}
where $\left(\langle \nabla u,\sigma'(s)\rangle-1\right)_-$ is the negative part of $\left(\langle \nabla u,\sigma'(s)\rangle-1\right)$.

Then by \eqref{Jlm3.7eq8} and \eqref{Jlm3.7eq9}, we have
\begin{align}\label{Jlm3.7eq10}
\left|\int_0^{d(\tilde y,\tilde z)}\left(\langle \nabla u,\sigma'(s)\rangle-1\right)_-ds\right|
&\le \left|\int_0^{d(\tilde y,\tilde z)}\left(\langle \nabla u,\sigma'(s)\rangle-1\right)_+-\left(\langle \nabla u,\sigma'(s)\rangle-1\right)_-ds\right|
+\left|\int_0^{d(\tilde y,\tilde z)}\left(\langle \nabla u,\sigma'(s)\rangle-1\right)_+ds\right|\notag\\
&\le \Psi(\delta|r,n).
\end{align}

Then by \eqref{Jlm3.7eq8} and \eqref{Jlm3.7eq10} we have
\begin{align}\label{Jlm3.7eq11}
\int_0^{d(\tilde y,\tilde z)}\left|\langle \nabla u,\sigma'(s)\rangle-1\right|ds
=\left|\int_0^{d(\tilde y,\tilde z)}\left(\langle \nabla u,\sigma'(s)\rangle-1\right)_++\left(\langle \nabla u,\sigma'(s)\rangle-1\right)_-ds\right|\le \Psi(\delta|r,n).
\end{align}

Thus we have
\begin{align}\label{Jlm3.7eq12}
|u(\sigma(s))-u(\sigma(0))-s|
&=\left|\int_0^s\left(\langle \nabla u,\sigma'(s')\rangle-1\right)ds'\right|\notag\\
&\le \int_0^{d(\tilde y,\tilde z)}\left|\langle \nabla u,\sigma'(s)\rangle-1\right|ds\notag\\
&\le \Psi(\delta|r,n).
\end{align}

Given \eqref{Jlm3.7eq12}, we can look back to \eqref{Jlm3.7eq6} and get
\begin{align}\label{Jlm3.7eq13}
\frac{1}{2}d^2(y,z)
&=\int_0^{d(\tilde y,\tilde z)}sds\pm\Psi(\delta|r,n)\notag\\
&=\int_0^{d(\tilde y,\tilde z)}\left(u(\sigma(s))-u(\sigma(0))\right)ds\pm\Psi(\delta|r,n)\notag\\
&=\int_U \left(u(\tau_s(l(s)))-u(\tau_s(0))\right)ds\pm\Psi(\delta|r,n).
\end{align}

By the fundamental theorem of calculus we have
\begin{align}\label{Jlm3.7eq14}
\int_U \left(u(\tau_s(l(s)))-u(\tau_s(0))\right)ds\pm\Psi(\delta|r,n)=\int_U\int_0^{l(s)} \langle \nabla u(\tau_s(t)),\tau_s'(t))\rangle dtds\pm\Psi(\delta|r,n).
\end{align}

By the fundamental theorem of calculus again we have
\begin{align}\label{Jlm3.7eq15}
\left|\int_0^{l(s)}\left(\langle \nabla u(\tau_s(t)),\tau_s'(t))\rangle -\langle \nabla u(\tau_s(l(s))),\tau_s'(l(s))\rangle\right) dt\right|
&=\left|\int_0^{l(s)}\int_t^{l(s)}\nabla^2 u\left(\tau_s'(s'),\tau_s'(s')\right)ds'dt\right|\notag\\
&\le \int_0^{l(s)}\int_0^{l(s)}|\nabla^2 u\left(\tau_s'(s'),\tau_s'(s')\right)|ds'dt\notag\\
&= l(s)\int_0^{l(s)}|\nabla^2 u\left(\tau_s'(s'),\tau_s'(s')\right)|ds'\notag\\
&\le l(s)\int_0^{l(s)}|\nabla^2 u(\tau_s(t)|dt.
\end{align}

Recall that by \eqref{hessian estm} we have $\int_U l(s)\int_0^{l(s)}|\nabla^2 u(\tau_s(t)|dtds\le \Psi(\delta|r,n)$, thus combining \eqref{Jlm3.7eq13}-\eqref{Jlm3.7eq15} we have
\begin{align}\label{Jlm3.7eq16}
\frac{1}{2}d^2(y,z)
&=\int_U\int_0^{l(s)} \langle \nabla u(\tau_s(l(s))),\tau_s'(l(s))\rangle dtds+l(s)\int_U\int_0^{l(s)}|\nabla^2 u(\tau_s(t)|dtds+\Psi(\delta|r,n)\notag\\
&= \int_U l(s)\langle \nabla u(\tau_s(l(s))),\tau_s'(l(s))\rangle ds+\Psi(\delta|r,n)\notag\\
&= \int_U l(s)\langle \nabla u(\sigma(s)),\tau_s'(l(s))\rangle ds+\Psi(\delta|r,n).
\end{align}

On the other hand, we estimate $d^2(x,y)-d^2(x,z)$ in terms of $u$:
\begin{align}\label{Jlm3.7eq17}
\frac{1}{2}d^2(x,y)-\frac{1}{2}d^2(x,z)
&=\frac{1}{2}d^2(\tilde x,\tilde y)-\frac{1}{2}d^2(\tilde x,\tilde z)\pm\Psi(\delta|r,n)\notag\\
&=\frac{1}{2}l^2(d(\tilde y,\tilde z))-\frac{1}{2}l^2(0)\pm\Psi(\delta|r,n)\notag\\
&=\int_U l'(s)l(s)ds\pm\Psi(\delta|r,n).
\end{align}

By the first variation formula of arc lenth, \eqref{Jlm3.7eq17} gives
\begin{align}\label{Jlm3.7eq18}
\frac{1}{2}d^2(x,y)^2-\frac{1}{2}d^2(x,z)^2
&=\int_U l(s)\langle \nabla u(\sigma(s)),\tau_s'(l(s))\rangle ds\pm\Psi(\delta|r,n).
\end{align}

Combining  \eqref{Jlm3.7eq16} and \eqref{Jlm3.7eq18}, we have
\begin{align}
d^2(y,z)=d^2(x,y)-d^2(x,z)\pm\Psi(\delta|r,n),
\end{align}
which completes the proof of the lemma.
\end{proof}

Following the aproach of Cheeger-Colding, we need the following lemmas:
\begin{lemma}[Lemma 1.14 in \cite{Ch97}]\label{TBineq}
Let $(M,g)$ be a complete Riemannian manifold and $u\in C^\infty(M)$. For any $x\in B_r(x), r,l>0$ and $0\le t\le l$ there holds
\begin{align}
\frac{1}{{\text{Liouv}}_g(SB_r(x))}\int_{SB_r(x)}\left|(u\circ\gamma_v)'(t)-\frac{u(\gamma_v(l))-u(\gamma_v(0))}{l}\right|d{\text{Liouv}}_g\le \frac{2l}{\Vol_g(B_r(x))}\int_{B_{r+l}(x)}|\nabla^2 u|d\Vol_g,
\end{align}
where $SB_r(x)$ is the unit tangent bundle of $B_r(x)$, ${\text{Liouv}}_g$ is the Liouville measure on the tangent bundle (which is defined as the product measure of the volume measure induced by $g$ with the spherical measure), $v\in SB_r(x))$ and $\gamma_v$ is the geodesic with $\gamma'(0)=v$.
\end{lemma}

\begin{lemma}\label{Jianglm3.8}
Let $(M^n,g,f)$ be a complete gradient shrinking Ricci soliton and $p$ be a minimum point of $f$. For any $u\in C^\infty(B_{100r}(p))$ satisfying (u1)-(u3), let $x,y\in B_{2r}(p)$ with $u(y)=a$, we have
\begin{align}
\left|d(x,u^{-1}(\{a\})-|u(x)-a|\right|<\Psi(\delta|r,n).
\end{align}
\end{lemma}

\begin{proof}
Without loss of generality we assume $u(x)=0<a$.  Let $\bar x$ be a point in $u^{-1}(\{a\})$ such that $d(x,\bar x)=d(x,u^{-1}(\{a\})$. By definition of $\bar x$ we have $d(x,\bar x)\le d(x,y)$. Since $x,y\in B_{2r}(p)$, we have
\begin{align}\label{lm3.8eq0}
d(x,u^{-1}(\{a\}))=d(x,\bar x)\le d(x,y)\le d(x,p)+d(y,p)< 4r.
\end{align}

Then we have
\begin{align}
d(\bar x,p)\le d(x,\bar x)+d(x,p)< 6r.
\end{align}

Let $\gamma_{x\bar x}$ be a minimal geodesic  from $x$ to $\bar x$, then we have
\begin{align}
d(z,p)\le d(x,z)+d(x,p)\le d(x,\bar x)+d(x,p) <6r, \forall z\in \gamma_{x\bar x}.
\end{align}

Thus by mean value theorem we have
\begin{align}\label{lm3.8eq1}
|u(x)-a|=|u(x)-u(\bar x)|\le \sup_{B_{6r}(p)}|\nabla u|d(x,\bar x).
\end{align}

By (u3), $\nabla u\le 1+\Psi(\delta|r,n)$, thus \eqref{lm3.8eq1} becomes
\begin{align}\label{lm3.8eq2}
|u(x)-a|=|u(x)-u(\bar x)|
&\le (1+\Psi(\delta|r,n))d(x,\bar x)\\ \notag
&= d(x,u^{-1}(\{a\}))+\Psi(\delta|r,n)d(x,u^{-1}(\{a\})).
\end{align}

Combining \eqref{lm3.8eq0} and \eqref{lm3.8eq2}, we have
\begin{align}\label{lm3.8eqc1}
|u(x)-a|-d(x,u^{-1}(\{a\}))\le \Psi(\delta|r,n)4r\le \Psi(\delta|r,n).
\end{align}

On the other hand,  using  Lemma \ref{TBineq} with its $l=5r$, we have for any $t\in [0,5r]$,
\begin{align}\label{lm3.8eq3}
\frac{1}{{\text{Liouv}}_g(SB_{2r}(x))}\int_{SB_{2r}(x)}\left|(u\circ\gamma_v)'(t)-\frac{u(\gamma_v(5r))-u(\gamma_v(0))}{5r}\right|d{\text{Liouv}}_g\le \frac{10r}{\Vol_g(B_{2r}(x))}\int_{B_{2r+5r}(x)}|\nabla^2 u|d\Vol_g.
\end{align}

By Lemma \ref{lm2.4} we have
\begin{align}
\frac{10r}{\Vol_g(B_{2r}(x))}\int_{B_{2r+5r}(x)}|\nabla^2 u|d\Vol_g \le \frac{C(n,r)}{\mu(B_{2r}(x))}\int_{B_{7r}(x)}|\nabla^2 u|d\mu.
\end{align}

Without loss of generality we assume $r>C_0(n)$ for the $C_0(n)$ in Lemma \ref{lm2.5}, then we have
 \begin{align}
\frac{\mu(B_{100r}(x))}{\mu(B_{2r}(x))}\le \frac{\mu(M)}{\mu(B_{2r}(x))}\le 2
\end{align}

By definition we have
 \begin{align}\label{lm3.8eq4}
\frac{C(n,r)}{\mu(B_{2r}(x))}\int_{B_{7r}(x)}|\nabla^2 u|d\mu\le \frac{C(n,r)\mu(B_{100r}(x))}{\mu(B_{2r}(x))}\fint_{B_{100r}(x)}|\nabla^2 u|d\mu.
\end{align}

Combining \eqref{lm3.8eq3}-\eqref{lm3.8eq4}, we have
\begin{align}\label{lm3.8eq5}
\fint_{SB_{2r}(x)}\left|(u\circ\gamma_v)'(t)-\frac{u(\gamma_v(5r))-u(\gamma_v(0))}{5r}\right|d{\text{Liouv}}_g\le C(n,r)\fint_{B_{100r}(x)}|\nabla^2 u|d\mu.
\end{align}

By (u1), \eqref{lm3.8eq5} becomes
\begin{align}\label{lm3.8eq6}
\fint_{SB_{2r}(x)}\left|(u\circ\gamma_v)'(t)-\frac{u(\gamma_v(5r))-u(\gamma_v(0))}{5r}\right|d{\text{Liouv}}_g\le \Psi(\delta|r,n).
\end{align}

By \eqref{lm3.8eq6} and (u2), we can choose $\tilde x\in B_{\Psi(\delta|r,n)}(x)$, and $v\in S_{\tilde x}M$, such that
\begin{align}\label{lm3.8eq7}
\left||\nabla u|(\tilde x)-1\right|\le \Psi(\delta|r,n),
\end{align}
\begin{align}\label{lm3.8eq8}
|v-\nabla u(\tilde x)|\le \Psi(\delta|r,n),
\end{align}
 and
\begin{align}\label{lm3.8eq13}
\left|\langle \nabla u,v\rangle-\frac{f(\gamma_v(5r))-f(\gamma_v(0))}{5r}\right|\le \Psi(\delta|r,n).
\end{align}

Since $v$ is a unit vector, by  Cauchy inequality and \eqref{lm3.8eq7} we have
\begin{align}\label{lm3.8eq9}
\left|\langle \nabla u,v\rangle\right| \le|\nabla u||v| \le 1+\Psi(\delta|r,n).
\end{align}

Since $v$ is a unit vector, by \eqref{lm3.8eq8} we have
\begin{align}\label{lm3.8eq10}
1+|\nabla u(\tilde x)|^2-2\langle \nabla u,v\rangle=|v-\nabla u(\tilde x)|^2\le \Psi(\delta|r,n),
\end{align}

Combining \eqref{lm3.8eq7} and \eqref{lm3.8eq10} we have
 \begin{align}\label{lm3.8eq11}
-2\langle \nabla u,v\rangle\le -1-|\nabla u(\tilde x)|^2+\Psi(\delta|r,n)\le -2+\Psi(\delta|r,n),
\end{align}

Combining \eqref{lm3.8eq9} and \eqref{lm3.8eq11}, we have
\begin{align}\label{lm3.8eq12}
\left|\langle \nabla u,v\rangle-1\right|^2&=\langle \nabla u,v\rangle^2+1-2\langle \nabla u,v\rangle\notag\\
&\le 1+1-2+\Psi(\delta|r,n)\le \Psi(\delta|r,n).
\end{align}

Thus by \eqref{lm3.8eq13} and \eqref{lm3.8eq12}, we have
\begin{align}\label{lm3.8eq14}
\left|f(\gamma_v(5r))-f(\gamma_v(0))-5r\right|\le \Psi(\delta|r,n).
\end{align}

By \eqref{lm3.8eq14}, for any $t\in[0,5r]$, we have
\begin{align}\label{lm3.8eq15}
u(\gamma_v(t))-u(\gamma_v(0))-t
&=\left(u\left(\gamma_v(5r)\right)-u\left(\gamma_v(0)\right)\right)-\left(u\left(\gamma_v(5r)\right)-u\left(\gamma_v(t)\right)\right)-t\notag\\
&\ge 5r-\Psi(\delta|r,n)-\left|u(\gamma_v(5r)-u(\gamma_v(t)\right|-t.
\end{align}

Similarly as \eqref{lm3.8eq1}, the minimal geodesics from $\gamma_v(t)$ to $\gamma_v(5r)$ are also in $B_{100r}(p)$, and by mean value theorem and (u3) we have
\begin{align}\label{lm3.8eq16}
\left|u(\gamma_v(5r)-u(\gamma_v(t)\right|&\le \sup_{B_{100r}(p)}|\nabla u|(5r-t)\notag\\
&\le (1+\Psi(\delta|r,n))(5r-t)\notag\\
&\le 5r-t+\Psi(\delta|r,n).
\end{align}

Combining \eqref{lm3.8eq15} and \eqref{lm3.8eq16}, we have
\begin{align}\label{lm3.8eq17}
u(\gamma_v(t))-u(\gamma_v(0))-t
&\ge 5r-t-\Psi(\delta|r,n)-\left|u(\gamma_v(5r)-u(\gamma_v(t)\right|\notag\\
&\ge -\Psi(\delta|r,n).
\end{align}

Recall $\gamma_v(0)=\tilde x\in B_{\Psi(\delta|r,n)}(x)$, similarly as \eqref{lm3.8eq1}, by mean value theorem we have
\begin{align}\label{lm3.8eq18}
|u(\tilde x)|=|u(x)-u(\tilde x)|\le \sup_{B_{100r}(p)}|\nabla u|d(x,\tilde x)\le \Psi(\delta|r,n).
\end{align}

By \eqref{lm3.8eq17} and \eqref{lm3.8eq18}, we have
\begin{align}\label{lm3.8eq19}
u(\gamma_v(t))=u(\gamma_v(t))-u(x)\ge u(\gamma_v(t))-u(\tilde x)-\Psi(\delta|r,n)\ge t-\Psi(\delta|r,n).
\end{align}

Similarly as \eqref{lm3.8eq1}, by mean value theorem we have
\begin{align}\label{lm3.8eq20'}
u(\gamma_v(t)=u(\gamma_v(t))-u(x)
&\le u(\gamma_v(t))-u(\gamma_v(0))+\Psi(\delta|r,n)\notag\\
&\le (1+\Psi(\delta|r,n))d(\gamma_v(t),x)+\Psi(\delta|r,n)\notag \\
&\le t+\Psi(\delta|r,n).
\end{align}

By \eqref{lm3.8eqc1} and \eqref{lm3.8eq0}, we have
\begin{align}\label{lm3.8eq20}
a=|u(x)-a|\le d(x,u^{-1}(\{a\}))+\Psi(\delta|r,n)\le 4r+\Psi(\delta|r,n)\le\frac{9}{2}r<5r,
\end{align}
for $\delta$ sufficiently small.

Thus by \eqref{lm3.8eq19}-\eqref{lm3.8eq20} and the intermediate value theorem we have that there is a $\hat t\in [a-\Psi(\delta|r,n),a+\Psi(\delta|r,n)]$ such that $u(\gamma_v(\hat t))=a$.

Let $t=\hat t$ in \eqref{lm3.8eq19}, we have
\begin{align}\label{lm3.8eq21}
a=u(\gamma_v(\hat t))\ge \hat t-\Psi(\delta|r,n)\ge d(\gamma_v(\hat t),\gamma_v(0))-\Psi(\delta|r,n).
\end{align}

Recall $\gamma_v(0)=\tilde x\in B_{\Psi(\delta|r,n)}(x)$, since $\gamma_
v(\hat t)\in u^{-1}(\{a\})$, we have
\begin{align}\label{lm3.8eq22}
d(\gamma_v(\hat t),\gamma_v(0))\ge d(\gamma_v(\hat t),x)-\Psi(\delta|r,n)\ge d(x,u^{-1}(\{a\}))-\Psi(\delta|r,n).
\end{align}

Combining \eqref{lm3.8eq21} and \eqref{lm3.8eq22}, we have
\begin{align}\label{lm3.8eq23}
|u(x)-a|-d(x,u^{-1}(\{a\}))=a-d(x,u^{-1}(\{a\}))\ge-\Psi(\delta|r,n).
\end{align}

By \eqref{lm3.8eqc1} and \eqref{lm3.8eq23}, we have
\begin{align}\label{lm3.8eq23'}
\left||u(x)-a|-d(x,u^{-1}(\{a\}))\right|\le\Psi(\delta|r,n),
\end{align}
which completes the proof of the lemma.
\end{proof}

\begin{lemma}\label{Jianglm3.9}
Let $(M^n,g,f)$ be a complete gradient shrinking Ricci soliton and $p$ be a minimum point of $f$. For any $u\in C^\infty(B_{100r}(p))$ satisfying (u1)-(u3),  $x\in B_{2r}(p)$ and $|s|\le 2$, there exists $y\in B_{5r}(p)$ such that $u(y)=u(x)+s$, and
\begin{align}
\left||u(y)-u(x)|-d(x,y)\right|\le \Psi(\delta|r,n).
\end{align}
\end{lemma}
\begin{proof}
The proof is similar as the proof of Lemma \ref{Jianglm3.8}, thus we only sketch it here. Without loss of generality we assume $u(x)=0$ and $s>0$.

It has been proved in the proof of Lemma \ref{Jianglm3.8} (see \eqref{lm3.8eq8}, \eqref{lm3.8eq7}, \eqref{lm3.8eq19} and \eqref{lm3.8eq20'}) that we can choose $\tilde x\in B_{\Psi(\delta|r,n)}(x)$, and $v\in S_{\tilde x}M$, such that
\begin{align}\label{lm3.8eq31}
\left||\nabla u|(\tilde x)-1\right|\le \Psi(\delta|r,n),
\end{align}
\begin{align}\label{lm3.8eq32}
|v-\nabla u(\tilde x)|\le \Psi(\delta|r,n),
\end{align}
 and
 \begin{align}\label{lm3.8eq33}
t-\Psi(\delta|r,n)\le u(\gamma_v(t))\le (1+\Psi(\delta|r,n))d(\gamma_v(t),x)+\Psi(\delta|r,n)\le t+\Psi(\delta|r,n), \forall t\in[0,5r].
\end{align}

Since $s\in[0,2r]$, by the intermediate value theorem, there exists $t_0\in[s-\Psi(\delta|r,n),s+\Psi(\delta|r,n)]$ such that
\begin{align}\label{lm3.8eq34}
u(\gamma_v(t_0))=s.
\end{align}

Let $y=\gamma_v(t_0)$, and let $t=t_0$ in \eqref{lm3.8eq33}, we have
 \begin{align}\label{lm3.8eq35}
-\Psi(\delta|r,n)\le u(y)-t_0=u(y)-u(x)-t_0\le \Psi(\delta|r,n), \forall t\in[0,5r].
\end{align}

We also have by \eqref{lm3.8eq33} that
 \begin{align}\label{lm3.8eq36}
-\Psi(\delta|r,n)\le  d(\gamma_v(t_0),x)-t_0\le \Psi(\delta|r,n).
\end{align}

Thus by \eqref{lm3.8eq35} and \eqref{lm3.8eq36}, we have
\begin{align}\label{lm3.8eq37}
\left||u(y)-u(x)|-  d(y,x)\right|=\left||u(y)-u(x)|-  d(\gamma_v(t_0),x)\right|\le \Psi(\delta|r,n),
\end{align}
which completes the proof of the lemma.
\end{proof}

Using lemmas above, now we can prove Theorem \ref{thmap}. Firstly we prove a special case $k=1$.
\begin{proof}[Proof of Theorem \ref{thmap} for the special case $k=1$]
Let $(M^n,g,f)$ be a complete gradient shrinking Ricci soliton and $p$ be a minimum point of $f$. Suppose $\lambda_k(\Delta_f) \le \frac{1}{2}+\delta$ for some nonnegative integer $k$. Then by Lemma \ref{lm2.1}, for any $r\ge C_0(n)$, there exists $u\in C^\infty(B_{100r}(p))$ satisfying (u1)-(u3).

Note that we only require $u$ to satisfy (u1)-(u3), thus without loss of generality we assume $u(p)=0$. Let $X=u^{-1}(\{0\})$ be the metric subspace of $M$. For any $x\in B_{2r}(p)$, let $\bar x\in X$ be a point such that $d(x,\bar x)=d(x,X)$. Define a map
\begin{align}
F:B_{2r}(p)&\to X\times \mathbb{R},\notag\\
x&\mapsto (\bar x,u(x)).
\end{align}

We will prove that $F$ is a $\Psi(\delta|r,n)$-Gromov Hausdorff map from $B_{2r}(p)$ to $B_{2r}(\bar x,u(p))$.

For any $x,y\in B_{2r}(p)$, without loss of generality we assume $u(x)\ge u(y)\ge0$. let $\bar x,\bar y\in X$ be a point such that $d(x,\bar x)=d(x,X)$ and $d(y,\bar y)=d(y,X)$ respectively. By triangluar inequality we have
\begin{align}
d(\bar x,p)\le d(\bar x,x)+d(x,p)\le 2d(x,p)<4r.
\end{align}
Thus we have $\bar x\in B_{4r}(p)$. Similarly we also have $\bar y\in B_{4r}(p)$.

 Let $\gamma$ denote a minimal geodesic from $\bar y$ to $y$ and $y_0:=\gamma \cap u^{-1}(\{u(x)\})$. Since $\bar x,\bar y\in X=u^{-1}(\{0\})$, we have $u(\bar x)=u(\bar y)=0$, and by Lemma \ref{Jianglm3.8} we have
\begin{align}\label{Jpf3.6eq1}
|d(x,\bar x)-u(x)|=|d(x,u^{-1}(\{0\}))-|u(x)-0||\le \Psi(\delta|r,n),
\end{align}
and,
\begin{align}\label{Jpf3.6eq2}
|d(y,\bar y)-u(y)|=|d(y,u^{-1}(\{0\}))-|u(y)-0||\le \Psi(\delta|r,n).
\end{align}

Given \eqref{Jpf3.6eq2}, similar as the proof of \eqref{Jlm3.7eq12}, and recall $u(y_0)=u(x)$, $u(\bar y)=0$, we have
\begin{align}\label{Jpf3.6eq3}
|d(y,y_0)-(u(y)-u(x))|=|d(y,y_0)-(u(y)-u(y_0))|\le \Psi(\delta|r,n),
\end{align}
and
\begin{align}\label{Jpf3.6eq3'}
|d(y_0,\bar y)-u(x)|=|d(y_0,\bar y)-(u(y_0)-u(\bar y))|\le \Psi(\delta|r,n).
\end{align}

Applying Lemma \ref{Pyth} to triangles $\Delta_{x\bar y\bar x},\Delta_{x\bar yy_0}$ and $\Delta_{xyy_0}$, we have
\begin{align}\label{Jpf3.6eq4}
|d^2(\bar x,\bar y)+d^2(x,\bar x)-d^2(x,\bar y)|\le \Psi(\delta|r,n),
\end{align}
\begin{align}\label{Jpf3.6eq5}
|d^2(y_0,\bar y)+d^2(x,y_0)-d^2(x,\bar y)|\le \Psi(\delta|r,n),
\end{align}
\begin{align}\label{Jpf3.6eq6}
|d^2(y,y_0)+d^2(x,y_0)-d^2(x,y)|\le \Psi(\delta|r,n).
\end{align}

To prove that $F$ is $\Psi(\delta|r,n)$-isometric, by triangular inequality, we have
\begin{align}\label{Jpf3.6eq7}
&|d^2(x,y)-d^2(F(x),F(y)|\notag\\
=&|d^2(x,y)-d^2(\bar x,\bar y)-(u(y)-u(x))^2|\notag\\
\le &|-d^2(\bar x,\bar y)-d^2(x,\bar x)+d^2(x,\bar y)|+|-d^2(y,y_0)-d^2(x,y_0)+d^2(x,y)|\notag\\
&\ +|d^2(x,\bar x)-d^2(x,\bar y)+d^2(y,y_0)+d^2(x,y_0)-(u(y)-u(x))^2|\notag\\
\le &|-d^2(\bar x,\bar y)-d^2(x,\bar x)+d^2(x,\bar y)|+|-d^2(y,y_0)-d^2(x,y_0)+d^2(x,y)|\notag\\
&\ +|d^2(y_0,\bar y)+d^2(x,y_0)-d^2(x,\bar y)|\notag\\
&\ +|d^2(x,\bar x)+d^2(y,y_0)-d^2(y_0,\bar y)-(u(y)-u(x))^2|\notag\\
\le &|-d^2(\bar x,\bar y)-d^2(x,\bar x)+d^2(x,\bar y)|+|-d^2(y,y_0)-d^2(x,y_0)+d^2(x,y)|\notag\\
&\ +|d^2(y_0,\bar y)+d^2(x,y_0)-d^2(x,\bar y)|\notag\\
&\ +|d^2(y,y_0)-(u(y)-u(x))^2|+|d^2(x,\bar x)-d^2(y_0,\bar y)|\notag\\
\le &|d^2(\bar x,\bar y)+d^2(x,\bar x)-d^2(x,\bar y)|+|d^2(y,y_0)+d^2(x,y_0)-d^2(x,y)|\notag\\
&\ +|d^2(y_0,\bar y)+d^2(x,y_0)-d^2(x,\bar y)|\notag\\
&\ +|d^2(y,y_0)-(u(y)-u(x))^2|+|d^2(x,\bar x)-u^2(x)|+|d^2(y_0,\bar y)-u^2(x)|\notag\\
\le &|d^2(\bar x,\bar y)+d^2(x,\bar x)-d^2(x,\bar y)|+|d^2(y,y_0)+d^2(x,y_0)-d^2(x,y)|\notag\\
&\ +|d^2(y_0,\bar y)+d^2(x,y_0)-d^2(x,\bar y)|\notag\\
&\ +C(r)\left(|d(y,y_0)-(u(y)-u(x))|+|d(x,\bar x)-u(x)|+|d(y_0,\bar y)-u(x)|\right).
\end{align}

Using \eqref{Jpf3.6eq1}-\eqref{Jpf3.6eq6} derictly, \eqref{Jpf3.6eq7}gives
\begin{align}\label{Jpf3.6eq10}
|d^2(x,y)-d^2(F(x),F(y)|
\le \Psi(\delta|r,n),
\end{align}
which tells that $F$ is $\Psi(\delta|r,n)$-isometric.

Now we prove that $F$ is $\Psi(\delta|r,n)$-onto from $B_2(p)$ to $B_2(0,\bar p)$. Take arbitrary $(t,\tilde x)\in B_{2-\psi(\delta|r,n)}(0,\bar p)$, where $\psi(\delta|r,n)$ is a function to be determined later and it converges to $0$ as $\delta$ approaches to $0$ and $r,n$ are fixed. Without loss of generality we assume $t\ge 0$.  By Lemma \ref{Jianglm3.9}, there exists $x\in B_5(p)$ such that
\begin{align}\label{Jpf3.6eq8}
u(x)=u(\tilde x)+t=0+t=t,
\end{align}
and
\begin{align}\label{Jpf3.6eq9}
|d(x,\tilde x)-t|=|d(x,\tilde x)-|u(x)-u(\tilde x)||\le \Psi(\delta|r,n).
\end{align}

Let $\bar x\in u^{-1}(\{0\})$ such that $d(x,\bar x)=d(x,u^{-1}(\{0\}))$, then by Lemma \ref{Jianglm3.8} we have
\begin{align}\label{Jpf3.6eq11}
|d(x,\bar x)-t|=|d(x,\bar x)-|u(x)-u(\bar x)||\le \Psi(\delta|r,n).
\end{align}

Combining \eqref{Jpf3.6eq8}-\eqref{Jpf3.6eq11}, and by triangular inequality, we have
\begin{align}\label{Jpf3.6eq12}
d(\bar x,\tilde x)\le|d(x,\bar x)-d(x,\tilde x)|\le|d(x,\tilde x)-t|+|d(x,\bar x)-t|\le \Psi_1(\delta|r,n),
\end{align}
where $\Psi_1(\delta|r,n)$ is a function converging to $0$ as $\delta$ approaches to $0$ and $r,n$ are fixed.

To show that $z\in B_2(p)$,   apply Lemma \ref{Pyth} to the triangle $\Delta_{zp\tilde z}$ and we get
\begin{align}
|d(x,p)^2-d(p,\tilde x)^2-d(\tilde x,x)^2|\le \Psi_2(\delta|r,n).
\end{align}

Then we have
\begin{align}
d(x,p)^2\le|d(x,p)^2-d(p,\tilde x)^2-d(\tilde x,x)^2|+d(p,\tilde x)^2+d(\tilde x,x)^2\le \Psi_2(\delta|r,n)+2-\psi(\delta|r,n)+\Psi_1(\delta|r,n),
\end{align}
where $\Psi_2(\delta|r,n)$ is a function converging to $0$ as $\delta$ approaches to $0$ and $r,n$ are fixed.

Taking $\psi(\delta|r,n)=\Psi_1(\delta|r,n)+\Psi_2(\delta|r,n)$, we have
\begin{align}
d(x,p)^2\le 2.
\end{align}

Thus $x\in B_2(p)$. Moreover, by \eqref{Jpf3.6eq8} and \eqref{Jpf3.6eq12}, $F(x)=(\bar x, u(x))=(\bar x,t)$ and $d(F(x),(t,\tilde x))\le \Psi(\delta|r,n)$, thus $F$ is $\Psi(\delta|r,n)$-onto  from $B_2(p)$ to $B_2(0,\bar p)$.

Therefore $F$ is a $\Psi(\delta|r,n)$-Gromov Hausdorff map from $B_2(p)$ to $B_2(0,\bar p)$, which proves the theorem for the special case $k=1$.
\end{proof}

Now we prove Theorem \ref{thmap} in general case:
\begin{proof}[Proof of Theorem \ref{thmap}]
For any complete gradient shrinking Ricci soliton $(M^n,g,f)$, if $\lambda_k(\Delta_f) \le \frac{1}{2}+\delta$ for some nonnegative integer $k$, then by Lemma \ref{lm2.1}, we have functions $u_i(i=1,...,k)$  on $B_{100r}(p)$ satisfying (u1)-(u4).

Let $\ell\ge0$ be an integer such that there holds
\begin{align*}
d_{\text{GH}}(B_{2r}(p),B_{2r}(\bar p,\mathbf{u}(p)))<\epsilon,
\end{align*}
for some length space $X\times \mathbb{R}^\ell\ni(\bar p,\mathbf{u}(p))$. Since $u_i(i=1,...,k)$ satisfies (u1)-(u3), by the result of special case $k=1$, we have $\ell\ge 1$, and each $u_i$ is $\Psi(\delta|r,n)$-close in Gromov-Hausdorff sense to a linear function $l_i$ in $B_{2r}(\bar p,\mathbf{u}(p))\subset X\times \mathbb{R}^\ell$.

Suppose $\ell <k$, then at least one function of $l_i(i=1,...,k)$ is a  linear combination of other functions. Without loss of generality we assume $l_1=\sum_{\iota=2}^ka_\iota l_\iota$ for some  $a_\iota \in \mathbb{R}(i\neq k_0)$. Denote $v=u_1-\sum_{\iota=2}^ka_\iota u_\iota$, then we have
\begin{align}
\sup_{B_{2r}(p)}|v|\le \Psi(\delta|r,n).
\end{align}

By (u4), assuming $\delta<\frac{1}{4}$ without loss of generality, we have
\begin{align}
\Delta_f v\ge -(\frac{1}{2}+\delta)v\ge -v.
\end{align}

Then by gradient estimate  (Lemma \ref{Gradestm}), we have
\begin{align}\label{nabla v1}
\sup_{B_r(p)}|\nabla v|
\le \fint_{B_{\frac{3}{2}r}(p)}|\nabla v|^2d\mu
\le \fint_{B_{2r}(p)}v^2d\mu
\le \sup_{B_{2r}(p)}|v|
\le\Psi(\delta|r,n).
\end{align}

On the other hand, we will prove that $|\nabla v|^2$ is close to $2$ in the average integral sense, which contradicts to \eqref{nabla v1}. To prove that, we will show $\sum_{\iota=2}^ka_\iota ^2\le1\pm \Psi(\delta|r,n)$ firstly. A direct calculation gives
\begin{align}
|\sum_{\iota=2}^ka_\iota \nabla u_i|^2=|\nabla \sum_{\iota=2}^ka_\iota u_\iota|^2=|\nabla u_1-\nabla v|^2=|\nabla u_1|^2+|\nabla v|^2-2\langle \nabla u_1,\nabla v\rangle.
\end{align}

Thus by \eqref{nabla v1} and Cauchy inequality we have
\begin{align}
|\nabla u_1|^2-\Psi(\delta|r,n)(|\nabla u_1|+1)
\le|\sum_{\iota=2}^ka_\iota \nabla u_i|^2
\le |\nabla u_1|^2+\Psi(\delta|r,n)(|\nabla u_1|+1), {\text{on}} \ B_r(p).
\end{align}

Taking integral on $B_r(p)$, by (u2) and H\"older inequality we have
\begin{align}\label{norm a1}
1-\Psi(\delta|r,n)\le
\sum_{\iota,\varsigma=2}^ka_\iota a_\varsigma\fint_{B_r(p)}\langle\nabla u_\iota,\nabla u_\varsigma\rangle d\mu
=\fint_{B_r(p)}|\sum_{\iota=2}^ka_\iota \nabla u_\iota|^2d\mu
\le1+ \Psi(\delta|r,n).
\end{align}

By (u2), the matrix $(\fint_{B_r(p)}\langle\nabla u_\iota,\nabla u_\varsigma\rangle d\mu)_{(k-1)\times (k-1)}$ is $\Psi(\delta|r,n)$-close to the identity matrix for every entries, thus the least eigenvalue of $(\fint_{B_r(p)}\langle\nabla u_\iota,\nabla u_\varsigma\rangle d\mu)_{(k-1)\times (k-1)}$ is greater than $1-\Psi(\delta|r,n)$, and the greatest eigenvalue is less than $1+\Psi(\delta|r,n)$, thus we have
\begin{align}\label{norm a2}
(1-\Psi(\delta|r,n))\sum_{\iota=2}^ka_\iota ^2\le
\sum_{\iota,\varsigma=2}^ka_\iota a_\varsigma\fint_{B_r(p)}\langle\nabla u_\iota,\nabla u_\varsigma\rangle d\mu
\le(1+ \Psi(\delta|r,n))\sum_{\iota=2}^ka_\iota ^2.
\end{align}

Combining \eqref{norm a1} and \eqref{norm a2}, we have

\begin{align}\label{norm a3}
\frac{1-\Psi(\delta|r,n)}{1+\Psi(\delta|r,n)}\le \sum_{\iota=2}^ka_\iota ^2\le\frac{1+\Psi(\delta|r,n)}{1-\Psi(\delta|r,n)},
\end{align}

which gives
\begin{align}\label{norm a4}
1-\Psi(\delta|r,n)\le \sum_{\iota=2}^ka_\iota ^2\le1+\Psi(\delta|r,n).
\end{align}

Now using (u2) again, a direct calculation gives
\begin{align}\label{nabla v2}
\fint_{B_{r}(p)}\left||\nabla v|^2-2\right|d\mu
&\le \fint_{B_{r}(p)}\left||\nabla u_1|^2+\sum_{\iota=2}^ka_\iota ^2|\nabla u_\iota|^2-2\right|d\mu+C(n)\sup_{1\le i<j\le k}\fint_{B_r(p)}|\langle\nabla u_\iota,\nabla u_j\rangle|d\mu\notag\\
&\le \fint_{B_{r}(p)}\left||\nabla u_1|^2-1\right|d\mu+\sum_{\iota=2}^ka_\iota ^2\fint_{B_{r}(p)}\left||\nabla u_\iota|^2-1\right|d\mu+\Psi(\delta|r,n)\notag\\
&\le\Psi(\delta|r,n),
\end{align}
which contradicts with \eqref{nabla v1}, thus $\ell \ge k$ and the theorem is proved.
\end{proof}

\section{proof of Theorem \ref{mthm2}}

In order to prove Theorem \ref{mthmre1} and  Theorem \ref{mthmre2}, we firstly prove that the soliton $(M^n,g,f)$ is smoothly close to the model soliton on a compact subset. To be precisely, we have the following two lemmas:
\begin{lemma}\label{mthm2'}
For any $r,\epsilon>0$ and any nonnegative integer $\ell$, there exists $\delta=\delta(n,r,\epsilon,\ell)>0$, such that for any complete gradient shrinking Ricci soliton $(M^n,g,f)$,
if $\lambda_n(\Delta_f) \le \frac{1}{2}+\delta$, then there exists a a diffeomorphism $\Phi$ from $B_{r}(p)$ to a subset of $\mathbb{R}^n$, such that
\begin{align*}
\|g-\Phi^*\bar g_{\text{Euc};\mathbb{R}^n}\|_{C^\ell(B_r(p))}+\|f-\Phi^*\bar f_{\text{Euc};\mathbb{R}^n}\|_{C^\ell(B_r(p))}<\epsilon,
\end{align*}
where $p$ is a minimum point of $f$, $(\mathbb{R}^{n},\bar g_{\text{Euc};\mathbb{R}^n},\bar f_{\text{Gau};\mathbb{R}^n})$ is the Gaussian soliton, that is,  $\bar g_{\text{Euc};\mathbb{R}^n}$ is the Euclidean metric, and $\bar f_{\text{Gau};\mathbb{R}^n}=\frac{|x|^2}{4}$ on $\mathbb{R}^n$.
\end{lemma}

\begin{lemma}\label{mthm2}
For any $r,\epsilon,v>0$ and any nonnegative integer $\ell$, there exists $\delta=\delta(n,r,\epsilon,v,\ell)>0$, such that for any complete gradient shrinking Ricci soliton $(M^n,g,f)$ with
\begin{align}\Vol_g(B_1(p_\alpha))\ge v,
\end{align}
if $\lambda_{n-2}(\Delta_f) \le \frac{1}{2}+\delta$, then there exists a product soliton $(X,\bar g_X,\bar f_X)\times (\mathbb{R}^{n-2},\bar g_{\text{Euc};\mathbb{R}^{n-2}},\bar f_{\text{Gau};\mathbb{R}^{n-2}}):=(X\times \mathbb{R}^{n-2},\bar g_X+\bar g_{\text{Euc};\mathbb{R}^{n-2}},\bar f_X+\bar f_{\text{Euc};\mathbb{R}^{n-2}})$,  and a diffeomorphism $\Phi$ from $B_{r}(p)$ to a subset of $X\times \mathbb{R}^{n-2}$, such that
\begin{align*}
\|g-\Phi^*(\bar g_X+\bar g_{\text{Euc};\mathbb{R}^2})\|_{C^\ell(B_r(p))}+\|f-\Phi^*(\bar f_X+\bar f_{\text{Euc};\mathbb{R}^2})\|_{C^\ell(B_r(p))}<\epsilon,
\end{align*}
where $p$ is a minimum point of $f$,  $(X,\bar g_X,\bar f_X)$ is $(\mathbb{R}^{2},\bar g_{\text{Euc};\mathbb{R}^2},\bar f_{\text{Gau};\mathbb{R}^2})$ or the standard $\mathbb{S}^2$ or $\mathbb{RP}^2$  with constant $\bar f_X$, and $(\mathbb{R}^{n-2},\bar g_{\text{Euc};\mathbb{R}^{n-2}},\bar f_{\text{Gau};\mathbb{R}^{n-2}})$ is the Gaussian soliton.
\end{lemma}

To prove Lemma \ref{mthm2} and Lemma \ref{mthm2'}, we will use the following two lemmas, Lemma \ref{splitf} and Lemma \ref{propu}:
\begin{lemma}\label{splitf}
Let $(X\times \mathbb{R}^{k}, \bar g_X+\bar g_{\text{Euc};\mathbb{R}^k})$ be the product Riemannian manifold of a Riemannian manifold $(X, \bar g_X)$ with the  Euclidean space $(\mathbb{R}^{k},\bar g_{\text{Euc};\mathbb{R}^k})$ for any positive integer $k$. Let $\bar u_i(i=1,...,k)\in C^\infty(X\times \mathbb{R}^{k})$ such that $\mathbf{\bar u}=(\bar u_1,...,\bar u_k)$ is an Euclidean coordinate restricted on the submanifold $\{\bar x\}\times \mathbb{R}^k$ for any $\bar x\in X$, and $\bar u_i(i=1,...,k)$ are constants restricted on $X\times \{(\bar u_1,...,\bar u_k)\}$ for any $(\bar u_1,...,\bar u_k)\in\mathbb{R}^k$. Suppose all the $\bar u_i(i=1,...,k)$ are eigenfunctions of $\Delta_{\bar f}$, for some $\bar f\in C^\infty (X\times \mathbb{R}^{k})$, corresponding to the eigenvalue $\frac{1}{2}$, then
\begin{align}
\bar f(\bar x,\mathbf{\bar u})=\bar f(\bar x,0^k)+\frac{1}{4}|\mathbf{\bar u}|^2,
 \end{align}
for any $(\bar x,\mathbf{\bar u})\in X\times \mathbb{R}^k$, where $|\cdot|$ is the Euclidean norm.
\end{lemma}
\begin{proof}
 $\bar u_i(i=1,...,k)$ is an eigenfunction  of $\Delta_{\bar f}$, corresponding to the eigenvalue $\frac{1}{2}$, means that
 \begin{align}\label{lm5.2-1}
\Delta_{\bar g}\bar u_i-\langle \nabla_{\bar g} \bar f,\nabla_{\bar g} \bar u_i\rangle+\frac{1}{2}\bar u_i=0.
 \end{align}

Since $\bar g$ is the product metric which is the Euclidean metric restricted in $\mathbb{R}^k$ and $\bar u_i$ is a coordinate function on $\mathbb{R}^{k}$, we have that $\Delta_{\bar g}\bar u_i=0$.

Thus \eqref{lm5.2-1} gives
 \begin{align}
\frac{\partial f}{\partial \bar u_i}=\langle \nabla_{\bar g} \bar f,\nabla_{\bar g} \bar u_i\rangle=\frac{1}{2}\bar u_i.
 \end{align}

 Therefore we have
 \begin{align}
\bar f(\bar x,\mathbf{\bar u})=\bar f(\bar x,0)+\frac{1}{4}|\mathbf{\bar u}|^2,
 \end{align}
which proves the lemma.
\end{proof}

\begin{lemma}\label{propu}
The functions $u_i(i=1,...,k)$ in Lemma \ref{lm2.1} satisfy
\begin{align}
\sup_{B_r(p)}|\nabla^m u_i|<C(n,r,m), \forall m=0,1,2,...
\end{align}
where $C(n,r,m)$ is a positive constant depending only on $n$, $r$ and $m$.
\end{lemma}
\begin{proof}
The functions $u_i(i=1,...,k)$ in Lemma \ref{lm2.1} are actually constructed in Lemma \ref{lm2.2}. By Lemma \ref{lm2.2} (1), we have
\begin{align}
\fint_M  u_i^2 d\mu=\lambda_i(\Delta_f)^{-1},\forall i=1,...,k,
\end{align}
where $\lambda_i(\Delta_f)$ is the eigenvalue corresponding to $u_i$.

It is known in    \cite[Theorem 1.1]{CZ17} (see also \cite{Mo05}, \cite{HN14}) that $\lambda_i(\Delta_f)\ge \frac{1}{2}$, thus we have
\begin{align}
\fint_M  u_i^2 d\mu\le 2,\forall i=1,...,k.
\end{align}

By Lemma \ref{lm2.5}, for any $r\ge C_0(n)$, we have
\begin{align}
\fint_{B_{2r}(p)}u_i^2d\mu\le \frac{\mu(M)}{\mu(B_{2r}(p))}\fint_M u_i^2d\mu\le 4, \forall r\ge C_0(n).
\end{align}

Note $u_i$ satisfies the eigenvalue equation $\Delta_f u_i+\lambda_i(\Delta_f) u_i=0$. Since for gradient shrinking Ricci solitons we have  Bochner formula, Moser iteration and gradient estimate (see \eqref{Bochner formula}, Lemma \ref{Msit} and Lemma \ref{Gradestm}),  the lemma follows by elliptic estimate (see Section 2).
\end{proof}

\begin{proof}[Proof of Lemma \ref{mthm2'}]
Suppose the lemma is not true, then for some $r_0,\epsilon_0>0$, and nonnegative integer $\ell_0$, there exists a sequence of real numbers $\{\delta_\alpha\}_{\alpha=1}^\infty$ satisfying $\lim_{\alpha\to\infty}\delta_\alpha=0$, and there exists a complete gradient shrinking Ricci soliton $(M_\alpha,g_\alpha,f_\alpha)$, whose drifted Laplace $\Delta_{f_\alpha}$ has the $n^{\text{th}}$ smallest nonzero eigenvalue  less than $\frac{1}{2}+\delta_\alpha$, such that there does not exist a diffeomorphism from $(B_{r_0}(p_\alpha),g_\alpha)$ to a subset of the shrinking Gaussian solition $(\mathbb{R}^{n},\bar g_{{\text{Euc}};\mathbb{R}^{n}},\bar f_{{\text{Gau}};\mathbb{R}^{n}})$, such that
\begin{align}\label{ctdc'}
\|g_\alpha-\Phi^*\bar g_{{\text{Euc}};\mathbb{R}^{n}}\|_{C^{\ell_0}(B_{r_0}(p_\alpha))}+\|f_\alpha-\Phi^*\bar f_{{\text{Gau}};\mathbb{R}^{n}}\|_{C^{\ell_0}(B_{r_0}(p_\alpha))}<\epsilon_0,
\end{align}
where $p_\alpha$ is a minimal point of $f_\alpha$.

 By Theorem \ref{thm1.2}, $\{(M_\alpha,g_\alpha,p_\alpha)\}_{\alpha=1}^\infty$ converges to the standard Euclidean space $(\mathbb{R}^n,\bar g_{{\text{Euc}};\mathbb{R}^n})$ in pointed Gromov-Hausdorff sense.

Therefore, since $(M_\alpha,g_\alpha,f_\alpha,p_\alpha)$ satisfy the shrinking soliton equation $\Ric_{g_\alpha}+\nabla_{g_\alpha}^2 f_\alpha=\frac{1}{2}g_\alpha$, passing to a subsequence, the pointed Gromov-Hausdorff convergence could be improved to be pointed smoothly convergence. There exists a smooth function $\bar f$ on $\mathbb{R}^n$, such that  for any $r>0$, there exists a domain $\Omega\supset B_r(0^n)\subset \mathbb{R}^n$ and embeddings $\Phi_\alpha:\Omega\to M_\alpha$ for large $\alpha$ such that $\Phi_\alpha(0^n)=p_\alpha$, $\Phi_\alpha(\Omega)\supset B_{p_\alpha}(r)$, $\Phi_\alpha^* g_\alpha$ and $\Phi_\alpha^* f_\alpha$ converge to $\bar g_{{\text{Euc}};\mathbb{R}^n}$ and $\bar f$ respectively on $\Omega$ in $C^\infty$ sense. Moreover, $(\mathbb{R}^n,\bar g_{{\text{Euc}};\mathbb{R}^n},\bar f)$ is also a shrinking soliton.

Now we prove that $\bar f=\bar f_{{\text{Gau}};\mathbb{R}^n}=\frac{1}{4}|\mathbf{\bar u}|^2+c$ for any $\mathbf{\bar u}\in \mathbb{R}^n$, where $|\cdot|$ is the Euclidean norm.

To prove that, for each $(M_\alpha,g_\alpha,f_\alpha,p_\alpha)$, we apply Lemma \ref{lm2.1} and get the eigenfunctions $u_{\alpha;i}(i=1,...,n)$, by lemma \ref{propu}, we have
\begin{align}
\sup_{B_r(p_i)}|\nabla^m u_{\alpha;i}|<C(n,r,m), \forall m=0,1,2,...
\end{align}

Thus by passing to a subsequence  there exists functions $\bar u_i(i=1,...,n)$ such that $\Phi_\alpha^*u_{\alpha;i}$ converge to $\bar u_i$ on $\Omega$  in the $C^\infty$ topology.

By Lemma \ref{lm2.1}, $u_{\alpha;i}$ satisfies
\begin{enumerate}
\item $\Delta_{f_\alpha}+(\frac{1}{2}+\delta_\alpha)u_{\alpha;i}=0,$
\item $\fint_{B_r(p)}|\langle\nabla u_{\alpha;i},\nabla u_{\alpha;j}\rangle-\delta_{ij}| d\mu< \Psi(\delta_\alpha|n,r), \forall r\ge C_0(n)$,
\item $\fint_{B_r(p)}|\nabla^2 u_{\alpha;i}|^2 d\mu< \Psi(\delta_\alpha|n,r)$.
\end{enumerate}

Since $\Phi_\alpha^*g_\alpha$, $\Phi_\alpha^*f_\alpha$ and $\Phi_\alpha^*u_{\alpha;i}$ are convergent to $\bar g, \bar f, \bar u_i$  in the $C^\infty$ topology respectively, by taking limit we have
\begin{enumerate}
\item[(1$'$)] $\Delta_{\bar f}+\frac{1}{2}\bar u_i=0,$
\item[(2$'$)] $\langle\nabla \bar u_i,\nabla \bar u_j\rangle-\delta_{ij}=0$, {\text{on}} $\mathbb{R}^n$,
\item[(3$'$)] $\nabla^2 \bar u_i=0$, {\text{on}} $\mathbb{R}^n$.
\end{enumerate}

By (2$'$) and (3$'$), we have that $\mathbf{\bar u}=(\bar u_1,...,\bar u_{n})$ is the Euclidean coordinates on $\mathbb{R}^{n}$. Thus by (1$'$) and Lemma \ref{splitf}, we have
\begin{align}\label{splitf'}
\bar f(\bar x,\mathbf{\bar u})=c+\frac{1}{4}|\mathbf{\bar u}|^2,
 \end{align}
 where $c$ is a constant.

 Thus $\bar f=\bar f_{{\text{Gau}};\mathbb{R}^n}$. Then we have $\Phi_\alpha^* g_\alpha$ and $\Phi_\alpha^* f_\alpha$ converge to $\bar g_{{\text{Euc}};\mathbb{R}^n}$ and $\bar f_{{\text{Gau}};\mathbb{R}^n}$ respectively on $\Omega$ in $C^\infty$ sense. 

 Take $r>2r_0$ and thus $\Phi_\alpha(\Omega)\supset B_{r_0}(p_\alpha)$ for $\alpha$ sufficently large, then \eqref{ctdc'} would be false for  $\alpha$ sufficently large, which is a contradiction and thus the lemma is proved.
\end{proof}

\begin{proof}[Proof of Lemma \ref{mthm2}]
Suppose the lemma is not true, then for some $r_0,\epsilon_0,v_0>0$, and nonnegative integer $\ell_0$, there exists a sequence of real numbers $\{\delta_\alpha\}_{\alpha=1}^\infty$ satisfying $\lim_{\alpha\to\infty}\delta_\alpha=0$, and there exists a complete gradient shrinking Ricci soliton $(M_\alpha,g_\alpha,f_\alpha)$, whose drifted Laplace $\Delta_{f_\alpha}$ has the $(n-2)^\text{th}$ smallest nonzero eigenvalue less than $\frac{1}{2}+\delta_\alpha$, such that
\begin{align}\label{nclps}
\Vol_{g_\alpha}(B_1(p_\alpha)\ge v_0,
\end{align}
 for any product soliton $(X\times \mathbb{R}^{n-2},\bar g,\bar f)=(X\times \mathbb{R}^{n-2},\bar g_X+\bar g_{{\text{Euc}};\mathbb{R}^{n-2}},\bar f_X+\bar f_{{\text{Gau}};\mathbb{R}^{n-2}})$, 
where $(X,\bar g_X,\bar f_X)$ is the shrinking Gaussian solition $(\mathbb{R}^{2},\bar g_{\text{Euc};\mathbb{R}^2},\bar f_{\text{Gau};\mathbb{R}^2})$ or the standard $\mathbb{S}^2$ or $\mathbb{RP}^2$  with constant $\bar f_X$, and $(\mathbb{R}^{n-2},\bar g_{{\text{Euc}};\mathbb{R}^{n-2}},\bar f_{{\text{Gau}};\mathbb{R}^{n-2}})$ is the shrinking Gaussian solition,  there does not exist a diffeomorphism from $(B_{r_0}(p_\alpha),g_\alpha)$ to a subset of $X\times \mathbb{R}^{n-2}$, such that
\begin{align}\label{ctdc}
\|g_\alpha-\Phi^*\bar g\|_{C^{\ell_0}(B_{r_0}(p_\alpha))}+\|f_\alpha-\Phi^*\bar f\|_{C^{\ell_0}(B_{r_0}(p_\alpha))}<\epsilon_0,
\end{align}
where $p_\alpha$ is a minimal point of $f_\alpha$.

We will deduce the contradiction in the following three steps.

\textbf{Step A:} In this step, by passing to a subsequence we will prove that $\{(M_\alpha,g_\alpha,f_\alpha,p_\alpha)\}_{\alpha=1}^\infty$ converges to a complete shrinking Ricci soliton $(X\times \mathbb{R}^{n-2},\bar g,\bar f,(\bar p,0^{n-2}))$ in the $C^\infty$ topology, where $\bar g=\bar g_X+\bar g_{{\text{Euc}};\mathbb{R}^{n-2}}$. It is an application of the codimensional $4$ theorem for the singular part in \cite{ZZ19}, see also \cite{ChNa15}.

 Since $(M_\alpha,g_\alpha,f_\alpha)$ are  shringking Ricci solitons,  by the volume comparison theorem (Lemma \ref{volcomp}) and Gromov's precompactness theorem, after passing to a subsequence, $\{(M_\alpha,g_\alpha,p_\alpha)\}_{\alpha=1}^\infty$ converges to some limit space $(\Sigma,\bar y)$ in pointed Gromov-Hausdorff sense, where $\Sigma$ is equipped with a metric $\bar d_\Sigma$.

Thus for any $\epsilon,r>0$, there exists $N_1$, such that
\begin{align}\label{Ys}
d_{\text{GH}}(B_{r}(p_\alpha),B_{r}(\bar y))<\epsilon/2,\forall \alpha\ge N_1.
\end{align}

On the other hand, by Theorem \ref{thm1.2}, after passing to a subsequence again, we have that there exists $N_2$, such that
\begin{align}\label{Ys1}
d_{\text{GH}}(B_{r}(p_\alpha),B_{r}(\bar p_\alpha,0^{n-2}))<\epsilon/2,\forall \alpha\ge N_2.
\end{align}
for some product length space $(\bar p_\alpha,0^{n-2})\in X_\alpha\times \mathbb{R}^{n-2}$, where $X_\alpha$ is equipped with a metric $\bar d_{X_\alpha}$ and $\mathbb{R}^{n-2}$ is equipped with the standard Euclidean metric.

Let $N=\max\{N_1,N_2\}$, then by \eqref{Ys}, \eqref{Ys1} and the triangular inequality, we have
\begin{align}
d_{\text{GH}}(B_{r}(\bar p_\alpha,0^{n-2}),B_{r}(\bar y))\le d_{\text{GH}}(B_{r}(p_\alpha),B_{r}(\bar y))+d_{\text{GH}}(B_{r}(p_\alpha),B_{r}(\bar p_\alpha,0^{n-2}))<\epsilon,\forall \alpha\ge N,
\end{align}
which tells that $\{(X_\alpha\times \mathbb{R}^{n-2},(\bar p_\alpha,0^{n-2}))\}_{\alpha=1}^\infty$ converges to $(\Sigma,\bar y)$ in pointed Gromov-Hausdorff sense.

Then, as subspaces, $\{(X_\alpha, \bar p_\alpha)\}_{\alpha=1}^\infty$ also converges to some limit space $(X,\bar p)$ in pointed Gromov-Hausdorff sense, where where $X$ is equipped with a metric $\bar d_X$.

Thus by taking product, we have that $\{(X_\alpha\times \mathbb{R}^{n-2},(\bar p_\alpha,0^{n-2}))\}_{\alpha=1}^\infty$ converges to $(X\times \mathbb{R}^{n-2},(\bar p,0^{n-2}))$ in pointed Gromov-Hausdorff sense.

By the uniqueness of the limit, we have that $(X\times \mathbb{R}^{n-2},(\bar p,0^{n-2}))$ and $(\Sigma,\bar y)$ are equivalent as pointed metric spaces.
Thus $\{(M_\alpha,g_\alpha,f_\alpha,p_\alpha)\}_{\alpha=1}^\infty$ converges to $(X\times \mathbb{R}^{n-2},(\bar p,0^{n-2}))$ in pointed Gromov-Hausdorff sense.

Since all the $(M_\alpha,g_\alpha)$ are complete, the limit space $X\times \mathbb{R}^{n-2}$ is also complete.

Recall that \eqref{fcontrol} gives
\begin{align}\label{nclps22}
0\le |\nabla f_\alpha|\le C(n,r), \text{on}\ B_r(p_\alpha), \forall r\ge 0.
\end{align}

Therefore, by \eqref{nclps22} and \eqref{nclps}, using \cite[Theorem 1.6 and Remark 1.7]{ZZ19}, we have that the singular set of the limit space $X\times \mathbb{R}^{n-2}$ has codimensional at least $4$, and thus we have that $X$ is a smooth manifold without boundary.

Therefore, since $(M_\alpha,g_\alpha,f_\alpha,p_\alpha)$ satisfy the shrinking soliton equation $\Ric_{g_\alpha}+\nabla_{g_\alpha}^2 f_\alpha=\frac{1}{2}g_\alpha$, passing to a subsequence, we have that the metric on $X\times \mathbb{R}^{n-2}$ is induced by some  smooth Riemannian metric $\bar g$ and there exists a smooth function $\bar f$ on $X\times \mathbb{R}^{n-2}$, such that  for any $r>0$, there exists a domain $\Omega\supset B_r(\bar p,0^{n-2})\subset X\times \mathbb{R}^{n-2}$ and embeddings $\Phi_\alpha:\Omega\to M_\alpha$ for large $\alpha$ such that $\Phi_\alpha(\bar p,0^{n-2})=p_\alpha$, $\Phi_\alpha(\Omega)\supset B_{p_\alpha}(r)$, $\Phi_\alpha^* g_\alpha$ and $\Phi_\alpha^* f_\alpha$ converge to $\bar g$ and $\bar f$ respectively on $\Omega$ in $C^\infty$ sense. Moreover, $X$ is of dimensional $2$, and $(X\times \mathbb{R}^{n-2},\bar g,\bar f)$ is also a shrinking soliton.

Therefore, as the metric on the subspace $X\subset X\times \mathbb{R}^{n-2}$, $\bar d_X$ is also induced by some smooth Riemannian metric $\bar g_X$, and we have $\bar g=\bar g_X+\bar g_{{\text{Euc}};\mathbb{R}^{n-2}}$, which completes the Step A.

\textbf{Step B:} In this step, we will prove that $\bar f$ is splitting, that is,
 \begin{align}\label{eqsplitf}
\bar f(\bar x,\mathbf{\bar u})=\bar f(\bar x,0^{n-2})+\frac{1}{4}|\mathbf{\bar u}|^2,
 \end{align}
for any $(\bar x,\mathbf{\bar u})\in X\times \mathbb{R}^{n-2}$, where $|\cdot|$ is the Euclidean norm.

To prove \eqref{eqsplitf}, for each $(M_\alpha,g_\alpha,f_\alpha,p_\alpha)$, we apply Lemma \ref{lm2.1} and get the eigenfunctions $u_{\alpha;i}(i=1,...,n-2)$, by lemma \ref{propu}, we have
\begin{align}
\sup_{B_r(p_i)}|\nabla^m u_{\alpha;i}|<C(n,r,m), \forall m=0,1,2,...
\end{align}

Thus by passing to a subsequence  there exists functions $\bar u_i(i=1,...,n-2)$ such that $\Phi_\alpha^*u_{\alpha;i}$ converge to $\bar u_i$ on $\Omega$  in the $C^\infty$ topology.

By Lemma \ref{lm2.1}, $u_{\alpha;i}$ satisfies
\begin{enumerate}
\item $\Delta_{f_\alpha}+(\frac{1}{2}+\delta_\alpha)u_{\alpha;i}=0,$
\item $\fint_{B_r(p)}|\langle\nabla u_{\alpha;i},\nabla u_{\alpha;j}\rangle-\delta_{ij}| d\mu< \Psi(\delta_\alpha|n,r), \forall r\ge C_0(n)$,
\item $\fint_{B_r(p)}|\nabla^2 u_{\alpha;i}|^2 d\mu< \Psi(\delta_\alpha|n,r)$.
\end{enumerate}

Since $\Phi_\alpha^*g_\alpha$, $\Phi_\alpha^*f_\alpha$ and $\Phi_\alpha^*u_{\alpha;i}$ are convergent to $\bar g, \bar f, \bar u_i$  in the $C^\infty$ topology respectively, by taking limit we have
\begin{enumerate}
\item[(1$'$)] $\Delta_{\bar f}+\frac{1}{2}\bar u_i=0,$
\item[(2$'$)] $\langle\nabla \bar u_i,\nabla \bar u_j\rangle-\delta_{ij}=0$, {\text{on}} $X\times \mathbb{R}^{n-2}$,
\item[(3$'$)] $\nabla^2 \bar u_i=0$, {\text{on}} $X\times \mathbb{R}^{n-2}$.
\end{enumerate}

By (2$'$) and (3$'$), we have that $\mathbf{\bar u}=(\bar u_1,...,\bar u_{n-2})$ is an Euclidean coordinate restricted on the submanifold $\{\bar x\}\times \mathbb{R}^{n-2}$ for any $\bar x\in X$, and $\bar u_i(i=1,...,n-2)$ are constants restricted on $X\times \{(\bar u_1,...,\bar u_{n-2})\}$ for any $(\bar u_1,...,\bar u_{n-2})\in\mathbb{R}^{n-2}$. Thus by (1$'$) and Lemma \ref{splitf}, we have
\begin{align}\label{splitf''}
\bar f(\bar x,\mathbf{\bar u})=\bar f(\bar x,0^{n-2})+\frac{1}{4}|\mathbf{\bar u}|^2,
 \end{align}
for any $(\bar x,\mathbf{\bar u})\in X\times \mathbb{R}^{n-2}$, where $|\cdot|$ is the Euclidean norm. Thus the Step B is completed.

\textbf{Step C:} In this step we will prove that $(X,\bar g_X,\bar f_X)$ is the shrinking Gaussian solition $(\mathbb{R}^{2},\bar g_{\text{Euc};\mathbb{R}^2},\bar f_{\text{Gau};\mathbb{R}^2})$ or the standard $\mathbb{S}^2$ or $\mathbb{RP}^2$  with constant $\bar f_X$, and complete the proof of the theorem.

For any vector field $V\in T(X\times \{0^{n-2}\})$, since $\bar g$ is the product metric of $\bar g_X$ with an Euclidean metric, we have
\begin{align}\label{subslt1}
\Ric_{\bar g_X}(V,V)=\Ric_{\bar g}(V,V),
\end{align}
where $\Ric_{\bar g_X}$ and $\Ric_{\bar g}$ are the Ricci tensor of $\bar g_X$ and $\bar g$ respectively.

Moreover, denote $\bar f_X(\cdot):=\bar f(\cdot,0^{n-2})$, then by \eqref{eqsplitf} we have $\bar f=\bar f_X+\bar f_{\text{Gau};\mathbb{R}^{n-2}}$, and we have
\begin{align}\label{subslt2}
\nabla^2_{\bar g_X}\bar f_X(V,V)=\nabla^2_{\bar g}\bar f(V,V),
\end{align}
where $\nabla^2_{\bar g_X}$ and $\nabla^2_{\bar g}$ are the Hessian with respect to $\bar g_X$ and $\bar g$ respectively.

Recall that $(X\times \mathbb{R}^{n-2},\bar g,\bar f)$ is a shrinking Ricci soliton, that is,
\begin{align}\label{subslt3}
\Ric_{\bar g}+\nabla^2_{\bar g}\bar f=\frac{1}{2}\bar g.
\end{align}

Thus combining \eqref{subslt1}, \eqref{subslt2} and \eqref{subslt3}, we have
\begin{align}
\Ric_{\bar g_X}(V,V)+\nabla^2_{\bar g_X}\bar f_{X}(V,V)=\frac{1}{2}\bar g_X(V,V),
\end{align}
for any vector field $V\in T(X\times \{0^{n-2}\})$.

Thus $(X,\bar g_X,\bar f_{X})$ is a shrinking Ricci soliton of dimensional $2$.

Since $(X\times \mathbb{R}^{n-2},\bar g)$ is complete, $(X,\bar g_{X})$ is also complete, thus by the classification of complete $2$- dimensional shrinking Ricci solitons (see \cite[Theorem 3.13]{Ch23}, see also \cite{CCZ08}), $(X,\bar g_X,\bar f_X)$ is the shrinking Gaussian solition $(\mathbb{R}^{2},\bar g_{\text{Euc};\mathbb{R}^2},\bar f_{\text{Gau};\mathbb{R}^2})$ or the standard $\mathbb{S}^2$ or $\mathbb{RP}^2$  with constant $\bar f_X$. For $\alpha$ sufficiently large, it  contradicts to our supposition at the beginning of the proof, thus the lemma is proved.
\end{proof}

Now we prove Theorem \ref{mthmre1} and Theorem \ref{mthmre2} one by one:
\begin{proof}[proof of Theorem \ref{mthmre1}]
By \cite[Theorem 9.2]{CM22}, l there exists an $R=R(n)$ such that if $(M^n,g,f)$ is a gradient shrinking Ricci soliton and $\{\bar f_{\text{Gau};\mathbb{R}^n}\le R\}\subset \mathbb{R}^n$ is close to $\{f\le R\}\subset M$ in the smooth topology and $\bar f_{\text{Gau};\mathbb{R}^n}$ and $f$ are close in the smooth topology on this set, then $(\mathbb{R}^n,\bar g_{\text{Euc};\mathbb{R}^n},\bar f_{\text{Gau};\mathbb{R}^n})$ and $(M^n,g,f)$ are identical after a diffeomorhpism.

From the proof of  \cite[Theorem 9.2]{CM22}, we can see that to conclude   $(\mathbb{R}^n,\bar g_{\text{Euc};\mathbb{R}^n},\bar f_{\text{Gau};\mathbb{R}^n})$ and $(M^n,g,f)$ are identical, the closeness in the smooth topology could be weakened to that in $C^{2,\alpha}$ topology. That is, if $\{\bar f_{\text{Gau};\mathbb{R}^n}\le R\}\subset \mathbb{R}^n$ is close to $\{f\le R\}\subset M$ in $C^{2,\alpha}$ and $\bar f_{\text{Gau};\mathbb{R}^n}$ and $f$ are close in $C^{2,\alpha}$ topology on this set, then we can deduce that $(\mathbb{R}^n,\bar g_{\text{Euc};\mathbb{R}^n},\bar f_{\text{Gau};\mathbb{R}^n})$ and $(M^n,g,f)$ are identical after a diffeomorhpism. 

In fact, in line 9 in the proof of \cite[Theorem 9.21]{CM22}, Colding-Minicozzi II says that $(\mathbb{R}^n,\bar g_{\text{Euc};\mathbb{R}^n},\bar f_{\text{Gau};\mathbb{R}^n})$ and $(M^n,g,f)$ are identical  after a diffeomorhpism if a proposition $(\dag_R)$ holds for some $R=R(n)$ sufficiently large. And by \cite[Theorem 6.1]{CM22}, the proposition $(\dag_R)$ could be deduced from their proposition $(\star_R)$. For readers' convenience, we quote their proposition $(\star_R)$ here (in the special case where their model soliton $\Sigma$ is setted to be $(\mathbb{R}^n,\bar g_{\text{Euc};\mathbb{R}^n},\bar f_{\text{Gau};\mathbb{R}^n})$):
\begin{itemize}
\item[$(\star_R)$] There is a diffeomorphism $\Psi_R$ from a subset of $\mathbb{R}^n$ to $M$ that is onto $\{b<R\}$ so that $|\bar g_{\text{Euc};\mathbb{R}^n}-\Psi_R^*g|^2+|\bar f_{\text{Gau};\mathbb{R}^n}-\Psi_R^*f|^2\le \delta_0 e^{\bar f_{\text{Gau};\mathbb{R}^n}-\frac{R^2}{4}}$ and $\|\bar g_{\text{Euc};\mathbb{R}^n}-\Psi_R^*g\|_{C^{2,\alpha}}^2+\|\bar f_{\text{Gau};\mathbb{R}^n}-\Psi_R^*f\|_{C^{2,\alpha}}^2\le \delta_0$.
	\end{itemize}

	Here their $b=2f^\frac{1}{2}$ , and $\delta_0=\delta_0(n)$ is a fixed small constant depending only on $n$.

Let $r=2 R(n)+2C_1(n)$, where $C_1(n)$ is given in our Lemma \ref{lm2.4}. 
Then by Lemma \ref{lm2.4} we have
\begin{align}
B_r(p)\supset \{b<2R(n)\}.
\end{align}

Thus let $r=2 R(n)+2C_1(n)$, $\epsilon=\delta_0(n) e^{-\frac{R(n)^2}{4}}$ and $\ell=3$ in Lemma \ref{mthm2'}, by Lemma \ref{mthm2'}, there exists $\delta=\delta(n)>0$, such that any complete gradient shrinking Ricci soliton $(M^n,g,f)$ with $\lambda_n(\Delta_f) \le \frac{1}{2}+\delta$ satisfies the proposition $(\star_{R(n)})$.

 Thus by \cite[Theorem 6.1]{CM22}, the proposition $(\dag_{R(n)})$ in \cite{CM22} holds. And by \cite{CM22}, $(M,g,f)$ is identical to  $(\mathbb{R}^n,\bar g_{\text{Euc};\mathbb{R}^n},\bar f_{\text{Gau};\mathbb{R}^n})$, which proves the theorem.
\end{proof}

The proof of  Theorem \ref{mthmre2} is nearly the same as that of Theorem \ref{mthmre2}:
\begin{proof}[proof of Theorem \ref{mthmre2}]
This time we use the strong rigidity of the product soliton $(X,\bar g_X,\bar f_X)\times (\mathbb{R}^{n-2},\bar g_{\text{Euc};\mathbb{R}^{n-2}},\bar f_{\text{Gau};\mathbb{R}^{n-2}}):=(X\times \mathbb{R}^{n-2},\bar g_X+\bar g_{\text{Euc};\mathbb{R}^{n-2}},\bar f_X+\bar f_{\text{Euc};\mathbb{R}^{n-2}})$, 
where $(X,\bar g_X,\bar f_X)$ is the shrinking Gaussian solition $(\mathbb{R}^{2},\bar g_{\text{Euc};\mathbb{R}^2},\bar f_{\text{Gau};\mathbb{R}^2})$ or the standard $\mathbb{S}^2$ or $\mathbb{RP}^2$  with constant $\bar f_X$, and $(\mathbb{R}^{n-2},\bar g_{\text{Euc};\mathbb{R}^{n-2}},\bar f_{\text{Gau};\mathbb{R}^{n-2}})$ is the shrinking Gaussian solition

Similarly as the proof of Theorem \ref{mthmre1}, by \cite{CM22}, $(X\times \mathbb{R}^{n-2},\bar g_X+\bar g_{\text{Euc};\mathbb{R}^{n-2}},\bar f_X+\bar f_{\text{Euc};\mathbb{R}^{n-2}})$ and $(M^n,g,f)$ are identical  after a diffeomorhpism if the following proposition $(\star'_R)$ holds for some $R=R(n)$ sufficiently large:
\begin{itemize}
\item[$(\star'_R)$] There is a diffeomorphism $\Psi_R$ from a subset of $X\times \mathbb{R}^{n-2}$ to $M$ that is onto $\{b<R\}$ so that $|\bar g_X+\bar g_{\text{Euc};\mathbb{R}^{n-2}}-\Psi_R^*g|^2+|\bar f_X+\bar f_{\text{Euc};\mathbb{R}^{n-2}}-\Psi_R^*f|^2\le \delta_0 e^{\bar f_X+\bar f_{\text{Euc};\mathbb{R}^{n-2}}-\frac{R^2}{4}}$ and $\|\bar g_X+\bar g_{\text{Euc};\mathbb{R}^{n-2}}-\Psi_R^*g\|_{C^{2,\alpha}}^2+\|\bar f_X+\bar f_{\text{Euc};\mathbb{R}^{n-2}}-\Psi_R^*f\|_{C^{2,\alpha}}^2\le \delta_0$.
	\end{itemize}

We will prove (2) firstly, thus  we suppose  $\lambda_{n-2}(\Delta_f)\le \frac{1}{2}+\delta$. 
Let $r=2 R(n)+2C_1(n)$, where $C_1(n)$ is given in our Lemma \ref{lm2.4}. 
Then by Lemma \ref{lm2.4} we have
\begin{align}
B_r(p)\supset \{b<2R(n)\}.
\end{align}

Thus letting $r=2 R(n)+2C_1(n)$, $\epsilon=\delta_0(n) e^{-\frac{R(n)^2}{4}}$ and $\ell=3$ in Lemma \ref{mthm2}, by Lemma \ref{mthm2}, there exists $\delta=\delta(n,v)>0$, such that any complete gradient shrinking Ricci soliton $(M^n,g,f)$ with
\begin{align}\Vol_g(B_1(p))\ge v,
\end{align} 
and $\lambda_n(\Delta_f) \le \frac{1}{2}+\delta$, where $p$ is a minimum point of $f$, satisfies the proposition $(\star'_{R(n)})$.

 Then, similarly as in the proof of Theorem \ref{mthmre1}, by \cite{CM22}, $(M,g,f)$ is identical to  $(\Sigma,\bar g,\bar f)$, which proves Theorem \ref{mthmre2} (2).

 For (1), suppose $\lambda_{n-1}(\Delta_f)\le \frac{1}{2}+\delta$. We will proof Theorem \ref{mthmre2} (1) by using Theorem \ref{mthmre2} (2). By definition, we have
 \begin{align}
\lambda_{n-2}(\Delta_f)\le \lambda_{n-1}(\Delta_f)\le\frac{1}{2}+\delta.
 \end{align}

 Thus by Theorem \ref{mthmre2} (2), $(M,g,f)$ is identical to $(X,\bar g_X,\bar f_X)\times (\mathbb{R}^{n-2},\bar g_{\text{Euc};\mathbb{R}^{n-2}},\bar f_{\text{Gau};\mathbb{R}^{n-2}}):=(X\times \mathbb{R}^{n-2},\bar g_X+\bar g_{\text{Euc};\mathbb{R}^{n-2}},\bar f_X+\bar f_{\text{Euc};\mathbb{R}^{n-2}})$, 
where $(X,\bar g_X,\bar f_X)$ is $(\mathbb{R}^{2},\bar g_{\text{Euc};\mathbb{R}^2},\bar f_{\text{Gau};\mathbb{R}^2})$ or the standard $\mathbb{S}^2$ or $\mathbb{RP}^2$  with constant $\bar f_X$.

 However, if   $(X,\bar g_X,\bar f_X)$ is the standard $\mathbb{S}^2$ or $\mathbb{RP}^2$  with constant $\bar f_X$, then the eigenvalues of $\Delta_{\bar f_X}$ is just the eigenvalues of the classical Laplacian on $\mathbb{S}^2$ or $\mathbb{RP}^2$, and it is known that their spectrum is
 \begin{align}
\lambda_1^{\mathbb{S}^2}=\lambda_1^{\mathbb{RP}^2}=2.
 \end{align}

 And the eigenvalues of $\Delta_f$ on  $(M,g,f)$ are $\lambda_i^{\mathbb{S}^2}+\lambda_j^{\mathbb{R}^{n-2}}(\Delta_{\bar f_{\text{Gau};\mathbb{R}^{n-2}}})$ or $\lambda_i^{\mathbb{RP}^2}+\lambda_j^{\mathbb{R}^{n-2}}(\Delta_{\bar f_{\text{Gau};\mathbb{R}^{n-2}}})$.

Recall
 \begin{align}
\lambda_1^{\mathbb{R}^{n-2}}(\Delta_{\bar f_{\text{Gau};\mathbb{R}^{n-2}}})=...=\lambda_{n-2}^{\mathbb{R}^{n-2}}(\Delta_{\bar f_{\text{Gau};\mathbb{R}^{n-2}}})=\frac{1}{2}, \lambda_{n-1}^{\mathbb{R}^{n-2}}(\Delta_{\bar f_{\text{Gau};\mathbb{R}^{n-2}}})=1.
 \end{align}

 Thus $(M,g,f)$ must have $\lambda_{n-1}(\Delta_f)\ge 1$ if $(X,\bar g_X,\bar f_X)$ is the standard $\mathbb{S}^2$ or $\mathbb{RP}^2$  with constant $\bar f_X$, which is a contradiction if we assume $\delta\le \frac{1}{2}$ without loss of generality. Thus we have $(X,\bar g_X,\bar f_X)=(\mathbb{R}^{2},\bar g_{\text{Euc};\mathbb{R}^2},\bar f_{\text{Gau};\mathbb{R}^2})$, and 
$(M,g,f)$ is identical to $(\mathbb{R}^{n},\bar g_{\text{Euc};\mathbb{R}^n},\bar f_{\text{Gau};\mathbb{R}^n})$, which proves (1). Note that (2) is proved above, before the proof of (1), thus  the theorem is proved.
\end{proof}

\section*{Acknowledgements}
This work is supported by  National Key R\&D Program of China: 2022YFA1005500.

Chang Li is supported by National Natural Science Foundation of China (Grant No. 12301079).

\bibliographystyle{plain}

 \end{document}